\documentclass[a4paper, 11pt]{amsart}
\pdfoutput=1
\usepackage{custom-template, shortcuts,graphicx}

\begin{document}

\makeatletter
\@namedef{subjclassname@2020}{%
	\textup{2020} Mathematics Subject Classification}
\makeatother

\title{Rank growth of elliptic curves over $N$-th root extensions}
\author{Ari Shnidman and Ariel Weiss}
\date{}
\address{Ari Shnidman, Einstein Institute of Mathematics, The Hebrew University of Jerusalem, Edmund J.\ Safra Campus, Jerusalem 9190401, Israel.\vspace*{-3pt}}
\email{ariel.shnidman@mail.huji.ac.il}
\address{Ariel Weiss, Department of Mathematics, Ben-Gurion University of the Negev, Be'er Sheva 8410501, Israel.\vspace*{-3pt}}
\email{arielweiss@post.bgu.ac.il}

\maketitle

\begin{abstract}
Fix an elliptic curve $E$ over a number field $F$ and an integer $n$ which is a power of $3$. We study the growth of the Mordell--Weil rank of $E$ after base change to the fields $K_d = F(\!\sqrt[2n]{d})$.  If $E$ admits a $3$-isogeny, then we show that the average ``new rank'' of $E$ over $K_d$, appropriately defined, is bounded as the height of $d$ goes to infinity.  When $n = 3$, we moreover show that for many elliptic curves $E/\Q$, there are no new points on $E$ over $\Q(\sqrt[6]d)$, for a positive proportion of integers $d$.  This is a horizontal analogue of a well-known result of Cornut and Vatsal. As a corollary, we show that Hilbert's tenth problem has a negative solution over a positive proportion of pure sextic fields $\Q(\sqrt[6]{d})$.  

The proofs combine our recent work on ranks of abelian varieties in cyclotomic twist families with a technique we call the ``correlation trick'',
which applies in a more general context where one is trying to show simultaneous vanishing of multiple Selmer groups.  We also apply this technique to families of twists of Prym  surfaces, which leads to bounds on the number of rational points in sextic twist families of bielliptic genus 3 curves.      
\end{abstract}

\section{Introduction}

Let $E$ be an elliptic curve over a number field $F$, and let $K/F$ be a finite extension. Mazur and Rubin define $E$ to be {\it diophantine stable for }$K/F$ if $E(K) = E(F)$, i.e.\ if there are no new rational points on $E$ after base change to $K$ \cite{MazurRubinDioStab}.  There has been much interest and speculation regarding how often $E$ is diophantine stable for $K/F$, as $K$ varies through a family of Galois extensions $K/F$ of fixed degree and Galois group $G$ 
 \cites{dokchitser,DFK,DarmonTian, kisilevsky,FKK, MazurRubinFindingLarge, MazurRubinGrowthSel, MazurRubinDioStab,Fornea,LemkeOliverThorne, kisilevsky-nam, beneish2021rank}.  Mazur and Rubin themselves showed that for a positive density set of primes $\ell$, the curve $E$ is diophantine stable for infinitely many $\Z/\ell^n\Z$-extensions $K/F$, under the mild hypothesis that $\End_F(E) = \End_{\bar F}(E)$.

In this paper, we study a more refined notion of ``new points''. Observe that there may be points $P \in E(K)$ defined over intermediate extensions $L/F$ contained in $K$.  Moreover, there may be points $P \in E(K)$ whose minimal field of definition is $K$, but which are sums of points defined over smaller fields. These types of points are not really new, so 
 we define the \emph{new part} of $E(K)$ to be the quotient
 \[E(K/F)\new := E(K)/\sum_LE(L),\] 
 where the sum is over the subfields $F\subset L\subsetneq K$.\footnote{As we explain in \Cref{sec:new-rank}, this definition is best behaved when $K$ contains no proper subextensions $L/F$ whose normal closure is also a normal closure for $K$. This will always be the case in the situations we consider.}
 Note that $E$ is diophantine stable for $K/F$ if and only if $E(L/F)^\mathrm{new} = 0$ for all extensions $F\subsetneq L$ contained in $K$.
 
 \subsection{Rank growth}
 
 Our first result shows that the average rank of $E(K_d/F)\new$ is bounded, for certain elliptic curves $E$ and for certain families of number fields of the form $K_d = F(\sqrt[2n]{d})$. In other words, there are not too many new points on average, as $|\Nm_{F/\Q}(d)| \to \infty$.

\begin{theorem}\label{thm:rankgainintro}
Let $E/F$ be an elliptic curve admitting a $3$-isogeny and fix a positive integer $n = 3^m$. As $d$ runs over the elements of $F\t/F^{\times2n}$ of order $2n$, ordered by height, the average of $\rk E(K_d/F)\new$ is bounded. \end{theorem}

For the precise definition of height see \Cref{def:height}. To prove \Cref{thm:rankgainintro}, we construct, for each field $K_d$, an abelian variety $A_d$ over $F$ whose Mordell--Weil rank is equal to the rank of $E(K_d/F)\new$. This uses a version of Serre's tensor construction, as in \cite{MRS}.
Each abelian variety $A_d$ admits a $\mu_n$-action, where $\mu_n = \langle \zeta_n\rangle$ is the group scheme of $n$-th roots of unity, and the $A_d$'s are all $\mu_{2n}$-twists of each other. Because $E$ has a 3-isogeny, the Galois module $A_1[1-\zeta_n]$ decomposes as a direct sum of characters. This allows us to apply our recent result \cite{ShnidmanWeiss}*{Thm.\ 1.1}, which bounds the average rank of abelian varieties in families of $\mu_{2n}$-twists.  As in \cite{ShnidmanWeiss}, the upper bound on the average rank that is guaranteed by \Cref{thm:rankgainintro} can be made explicit, but the bound depends on the particular curve $E$.

In the spirit of Mazur and Rubin, we are more interested in proving that $E(K_d/F)\new = 0$ for infinitely many $d$. Actually, we consider the harder question of whether this is true for a \emph{positive proportion} of $d \in F^\times/F^{\times2n}$, which is what we will need for our applications to Hilbert's tenth problem below.     When $n = 1$ and $F = \Q$, we have $\rk E(K_d/F)\new = \rk E_d(F)$, where $E_d$ is the $d$-th quadratic twist of $E$, and Goldfeld conjectures that $50\%$ of these twists have rank $0$.
This conjecture has been verified in many cases \cites{bkls, kriz-li, smith}. Indeed, in his Ph.D. thesis, Smith proves the conjecture for ``most'' elliptic curves $E$ over $\Q$ \cite{Smith-thesis}. 

In contrast, when $n > 1$, there may not be a single $d$ for which $\rk E(K_d/F)\new = 0$, due to root number considerations. This phenomenon was already observed by Dokchitser in \cite{dokchitser} for cubic twists. For example, if $n = 3$, the Birch and Swinnerton-Dyer conjecture implies that for the elliptic curve $E \colon y^2 +y= x^3 + x^2 + x$ over $\Q$ of conductor 19, the group $E(K_d/\Q)\new$ will have odd rank for all squarefree integers $d$.

Despite these somewhat pathological examples, we prove that for many elliptic curves, we indeed have $E(K_d/\Q)\new=0$ for a positive proportion of $d \in \Q\t/\Q^{\times 6}$. We consider elliptic curves with the model $E \colon y^2 + axy + by = x^3$, whose discriminant is $b^3(a^3 - 27b)$. These are precisely the elliptic curves over $\Q$ with a rational $3$-torsion point, which is $(0,0)$ in this model.

\begin{theorem}\label{thm:pos-prop}
    Let $a$ and $b$ be coprime integers, and let $E$ be the elliptic curve $y^2 + axy + by = x^3$. Assume that $3\nmid ab$ and that either:
    \begin{enumerate}
        \item there exists a prime $q\equiv 2\pmod 3$ such that $q\mid a^3 - 27b$, or
        \item there exist primes $q_1 \equiv 1\pmod3$ and $q_2\equiv 2\pmod 3$ such that $q_1\mid a^3-27b$ and $q_2\mid b$.
    \end{enumerate}
    Set $K_d = \Q(\sqrt[6]{d})$.  Then for a positive proportion of integers $d$, we have $\rk E(K_d/\Q)\new = 0$.
\end{theorem}

More precisely, the lower density of integers $d$ such that $\rk E(K_d/\Q)\new = 0$ is positive. Our proof gives an explicit lower bound on this lower density, but the bound depends on $E$ and tends to be quite small. For example, in \Cref{subsec:example}, we work out the details for the curve $E\: y^2 + 2xy - y = x^3$ of conductor $35$. In this case, we exhibit a set $T$ of squarefree integers, defined by finitely many congruence conditions, such that a proportion of at least $\frac{1}{18}$ of $d \in T$ satisfy $\rk E(K_d/\Q)\new = 0$.

In the setting of \Cref{thm:pos-prop}, the Galois group of the splitting field of $K_d/\Q$ is the dihedral group $D_6$ of order $12$, with the subgroup $C_6$ cut out by the imaginary quadratic field $\Q(\zeta_3)$. In this context, the groups $E(K_d/\Q)\new$ have a systematic source of rational points, namely $\chi_d$-components of Heegner points, where $\chi_d$ is the corresponding ring class field character of order 6. In terms of $L$-functions, the Birch and Swinnerton-Dyer conjecture predicts that the rank of $E(K_d/\Q)\new$ is 0 if and only if the twisted $L$-function $L(E, \chi_d, s)$ is non-vanishing at its central point $s = 1$.  \Cref{thm:pos-prop} should be compared to the non-vanishing results of Cornut--Vatsal \cite{cornut-vatsal}, who showed that for every prime $\p$ of $\O_K$, and for all large enough $n$, there exist ring class field characters $\chi$ of conductor $\p^n$ such that $L(E,\chi,1)$ is non-vanishing. Our result is orthogonal to theirs (``horizontal''  instead of ``vertical''), since we fix the order of the character while allowing many primes to divide the conductor.  Of course, we only consider a very special case, where $K = \Q(\zeta_3)$ and the order of the characters is 6.  Our general method could conceivably be adapted to other quadratic fields, but that seems to require some new ideas  (one would first of all need to generalize \cite{ShnidmanWeiss} appropriately).  

\subsection{The correlation trick}

The proof of \Cref{thm:pos-prop} uses \Cref{thm:rankgainintro} as a starting point, but requires significantly more. 
The abelian varieties $A_d$ from the proof of \Cref{thm:rankgainintro} are, in this case, abelian surfaces with multiplication by $\Z[\zeta_3]$. In fact, they are twists of the Jacobian of the genus two curve 
\[C \colon y^2 =x^6 + \alpha x^3 + 1,\]
where $\alpha = 108b/a^3-2$. Notice that $\Aut(C)$ contains the group $D_{6}$, and in particular the automorphism $\zeta_3(x,y) = (\zeta_3x,y)$ of order 3.\footnote{The non-hyperelliptic involution is  $(x :y : z) \to (z : y:x)$, when written in weighted projective coordinates.}   The endomorphism $2\zeta_3 + 1 = \sqrt{-3} \in \End(A_d)$ is only defined over $\Q(\sqrt{-3})$, but it descends to a $(3,3)$-isogeny over $\Q$, which factors into two 3-isogenies $\phi_d \colon A_d \to B_d$ and $\psi_d \colon B_d \to A_{-27d}$.

Under the mild technical conditions of \Cref{thm:pos-prop}, we show that there exists a positive density set $T \subset \Z$, defined by congruence conditions, such that $\#\Sel(\phi_d) = \#\Sel(\widehat{\phi}_d)$ and $\#\Sel(\psi_d) = \#\Sel(\widehat\psi_d)$ for all but finitely many $d \in T$. We reiterate that without some kind of technical condition on $E$, the parity conjecture implies that \Cref{thm:pos-prop} is false in general. When $d\in T$, we show that $\#\Sel_3(A_d) \leq (\#\Sel(\phi_d)\#\Sel(\psi_d))^2$. Since 
\[\rk E(K_d/\Q)\new =  \rk A_d(\Q) \leq  \dim_{\F_3}\Sel_3(A_d),\]
in order to prove \Cref{thm:pos-prop}, it is enough to show that $\Sel(\phi_d) = \Sel(\psi_d) = 0$ for a positive proportion of integers $d\in T$. 

Using the results of \cite{ShnidmanWeiss}, we show that the average size of each of $\Sel(\phi_d)$ and $\Sel(\psi_d)$ is $2$, for $d \in T$. Since these Selmer groups are $\F_3$-vector spaces, we immediately deduce that each of these groups vanishes for at least $50\%$ of $d\in T$. The problem is that, a priori, we cannot rule out the possibility that the sets $\{d \in T \colon \Sel(\phi_d) = 0\}$ and $\{ d\in T \colon \Sel(\psi_d) = 0\}$ intersect in a set of density 0.  Thus, to complete the proof, it remains to rule out the unlikely pathological scenario that half of $d \in T$ satisfy $\#\Sel(\phi_d) = 1$ and $\#\Sel(\psi_d) = 3$, and the other half satisfy $\#\Sel(\phi_d) = 3$  and $\#\Sel(\psi_d)=1$. In other words, we must show that the random variables $\#\Sel(\phi_d)$ and $\#\Sel(\psi_d)$ are at least a tiny bit correlated, e.g.\ that they are either both 1 or both greater than 1, for a positive proportion of $d \in T$.  

To prove this correlation, we use a third 3-isogeny $\eta_d \colon A_d \to C_d$, whose Selmer group $\Sel(\eta_d)$ can be interpreted as lying in the intersection of $\Sel(\phi_d)$ and $\Sel(\psi_d)$. We then apply the results of \cite{ShnidmanWeiss} to $\eta_d$ to show that the average size of the intersection is strictly greater than 1.  However, this \emph{does not} show that their intersection is non-trivial a positive proportion of the time, since the bulk of the average size could be supported on a 0-density set.  The trick is to observe that since 
\[\#\Sel(\phi_d) + \#\Sel(\psi_d) = \min(\#\Sel(\phi_d), \#\Sel(\psi_d)) + \max(\#\Sel(\phi_d), \#\Sel(\psi_d)), \]
we can infer that the average of their maximum size is strictly less than $2 + 2 - 1 = 3$. Since we are dealing with $\F_3$-vector spaces, this implies  that $\Sel(\phi_d)$ and $\Sel(\psi_d)$ are both trivial for a positive proportion of $d\in T$.
By the definition of $T$, we have $\Sel_3(A_d) = 0$ for such $d$, and $\rk(A_d) = 0$ as well, which proves the theorem. 

\subsection{Application to Hilbert's tenth problem for pure sextic fields}

Hilbert asked whether there is a Turing machine that takes as input a polynomial equation over $\Z$ and correctly decides whether it has a solution over $\Z$. Matiyasevich \cite{Matiyasevich}, building on work of Davis--Putnam--Robinson \cite{DPR}, showed that no such algorithm exists, i.e.\ Hilbert's tenth problem has a negative solution over $\Z$. 

There has been much work on the analogous question  over rings of integers $\O_K$ of number fields $K$ of degree larger than 1; see the introduction to \cite{garcia-fritz-pasten} for a brief survey. It is believed that the answer should again be negative, and this has been proven for number fields with at most one complex place. In particular, it is known when $K$ has degree 2 or 3, and for many $K$ of degree 4 as well.  While the general case is still open, Mazur and Rubin have shown that a negative answer follows from the finiteness of the Tate--Shafarevich group \cite{MazurRubinHilbert10}. This uses a result of Shlapentokh \cite{Shlapentokh} which states that Hilbert's tenth problem over $\O_K$ has a negative answer if there exists an elliptic curve $E/\Q$ of positive rank such that $\rk E(K) = \rk E(\Q)$, i.e.\ with no rank gain over $K$.  We combine this criterion with \Cref{thm:pos-prop} to prove:

\begin{theorem}\label{thm:hilbert}
 For a positive proportion of pure sextic fields $K = \Q(\sqrt[6]{d})$, ordered by the height of $d$, the analogue of Hilbert's tenth problem over $\O_K$ has a negative solution.   
\end{theorem}

Recently, Garcia-Fritz and Pasten \cite{garcia-fritz-pasten} proved a similar result for  an explicit subset of fields of the form $\Q(\sqrt[6]{-p^2q^3})$, with $p$ and $q$ prime. The set of all such fields has density 0 among all pure sextic fields, so \Cref{thm:hilbert} is a quantitative improvement.  On the other hand, our result does not give an explicit set of fields.  We prove \Cref{thm:hilbert} by applying \Cref{thm:pos-prop} to a single elliptic curve, but to maximize the proportion of pure sextic fields that our method gives, one could try to use many different elliptic curves. We do not attempt such an analysis here.

\subsection{Applications to twists of abelian surfaces}

The correlation trick can be applied in other settings where one needs to simultaneously bound  Selmer groups attached to two isogenies whose kernels are isomorphic (as group schemes). As an example, we prove the existence of simple abelian surfaces with low rank in some of the sextic twist families considered in \cite{ShnidmanWeiss}, partially answering a question we asked there.  Recall that if $C \to E$ is a ramified double cover of curves, then the Prym variety is the kernel of the induced map $\Jac(C) \to \Jac(E)$ on Jacobians.  As a special case, a smooth genus 3 curve of the form $C \colon y^3 = f(x^2)$, with $f(x)$ quadratic, admits a double cover to the elliptic curve $y^3 = f(x)$.  In this case, the Prym variety is an abelian surface. 

\begin{theorem}\label{thm:qm-pryms}
    Fix integers $a > b > 0$ and let $A_d$ be the twist family of Prym surfaces arising from the genus three bielliptic curves $C_d\: y^3 = (x^2- da^2)(x^2 - db^2)$.  For a positive proportion of integers $d$, ordered by absolute value, we have $\rk A_d \le 1$.
\end{theorem}

As a corollary, we prove: 
\begin{theorem}\label{thm:qmcurves}
    Let $C_d$ be any twist family as above. For a positive proportion of $d\in \Q\t/\Q^{\times 6}$, we have $\# C_d(\Q)\le 5$.
\end{theorem}

Individual cases of these theorems were proven in \cite{ShnidmanWeiss}*{\S1.2}; for other results on average Mordell--Weil ranks in families of Prym surfaces see \cite{laga} and \cite[Thm.\ 1.13]{AlpogeBhargavaShnidman}. As usual, one can extract from the proofs of \Cref{thm:qm-pryms} (resp.\  \Cref{thm:qmcurves}) an explicit lower bound on the proportion of twists with $\rk A_d \le 1$ (resp.\ $\#C_d \leq 5$), a bound which in principle depends on $a$ and $b$. For this particular family of curves, it is conceivable that our bound could be made uniform, independent of $a$ and $b$, since the abelian surfaces $A_d$ have everywhere potentially good reduction. To prove this, one would want to prove a version of Tate's algorithm for these surfaces and various isogenous surfaces, or find some other way to access their Tamagawa numbers. This is beyond the scope of this paper, but would be an interesting future project.  

Another large family of curves for which this intersection method applies is the family of genus two curves $C$ admitting {\it potential} $\sqrt{3}$-multiplication and a subgroup $(\Z/3\Z)^2 \hookrightarrow \Jac(C)[3]$ which is isotropic with respect to the Weil pairing. This family is parameterized by a twist of a certain Hilbert modular surface (since the $\sqrt{3}$-multiplication is not defined over $\Q$). An explicit rational parameterization for  this surface was given in \cite[\S2]{bruinflynnshnidman}, on the way to constructing a rational parameterization for the Hilbert modular surface itself. We will not prove any theorems about such abelian surfaces since the ideas are similar. In fact the proofs are easier in this case since these families only admit quadratic twists, not sextic twists.

\subsection{Outline}

We begin in \Cref{sec:encoding} by reinterpreting the notion of ``new rank'' representation-theoretically. Using this interpretation, we construct an abelian variety $B/F$ such that $\Z[\zeta_n]\sub \End_{\bar F}B$, and such that for $d\in H^1(F, \mu_{2n})$, the rank of the corresponding twist $B_d/F$ encodes the new rank of $E(K_d/F)$. Although this family of twists exactly encodes the new rank of $E$, the abelian variety $B$ does not satisfy the main hypothesis of \cite{ShnidmanWeiss}*{Thm.\ 1.1}. In \Cref{sec:avg-rank}, we construct an isogenous abelian variety $A/F$ that does satisfy this hypothesis. Hence, we may apply \cite{ShnidmanWeiss}*{Thm.\ 1.1} and deduce \Cref{thm:rankgainintro}.

In \Cref{sec:pos-prop}, we demonstrate in general how to use intersections of Selmer groups to show that the two Selmer groups $\Sel(\phi_d)$ and $\Sel(\psi_d)$ vanish simultaneously for a positive proportion of $d$. The key result of this section is \Cref{prop:ab-surface-pos-prop}, whose proof uses the correlation trick described above. Using this result, in \Cref{sec:proof-pos-prop}, we prove \Cref{thm:pos-prop}. In \Cref{subsec:example}, we show how to make \Cref{prop:ab-surface-pos-prop} quantitative, using the curve $E\:y^2 + 2xy - y = x^3$ as an example. In \Cref{sec:hilbert} we deduce \Cref{thm:hilbert} from \Cref{thm:pos-prop}. 
 Finally, in \Cref{sec:QM}, we apply \Cref{prop:ab-surface-pos-prop} in the context of Prym abelian surfaces to prove \Cref{thm:qm-pryms}.

\section{Encoding the new rank of $E$}\label{sec:encoding}

Let $F$ be a number field and let $E$ be an elliptic curve over $F$. The goal of this section is to construct, for each $d\in F\t/F^{\times 2n}$, an abelian variety $B_d$ over $F$, whose rank encodes the ``new rank'' of $E$ for the extension $K_d/F$.

First, for any extension $K/F$, we define a slight variant of the new rank, $\rk E(K/F)\gnew$,
which we then interpret from a  representation-theoretic point of view. When $K =K_d =  F(\!\sqrt[2n]d)$, it will turn out that $\rk E(K/F)\new = \rk E(K/F)\gnew$, so the results of this section will be relevant to the proof of \Cref{thm:rankgainintro}.

\subsection{The new rank of an elliptic curve}\label{sec:new-rank}

\subsubsection{Galois extensions}

Let $K/F$ be a finite Galois extension and let $G = \Gal(K/F)$. The group  $E(K)$ is a $G$-module, and we define a $G$-module quotient that corresponds to the ``new points''. 

\begin{definition}
    Set $E(K/F)\gnew = E(K)/\sum_LE(L)$, where the sum is over the subfields $F\subset L\subsetneq K$ that are Galois over $F$.
\end{definition}

Since we are concerned only with ranks (and not torsion properties), we will mostly consider the $G$-representation $V :=E(K)\tensor_\Z\Q$. We can decompose $V$ as a $\Q[G]$-module: $V = \bigoplus_\rho V_\rho$, where the sum is over all irreducible $\Q$-representations $\rho$ of $G$, and $V_\rho$ is the $\rho$-isotypic part of $V$, spanned by the images of all $G$-morphisms from $\rho$ to $V$. 

\begin{definition}
    Let $V\new$ denote the representation $\bigoplus_{\eta} V_\eta$, where the sum is over all \emph{faithful} irreducible $\Q$-representations $\eta$ of $G$.
\end{definition}

\begin{proposition}\label{prop:new-rank-equivalence}
    We have $E(K/F)\gnew\tensor_\Z\Q \simeq V\new$ and $\rk E(K/F)\gnew = \dim_\Q V\new$.
\end{proposition}

\begin{proof}
      Let $L$ be a proper subextension of $K$ that is Galois over $F$, and let $H = \Gal(K/L)$. If $W\sub V$ is any $G$-subrepresentation, then since $H$ is normal in $G$, the subspace $W^H\sub W$ of $H$ fixed vectors is a $G$-subrepresentation.  
      
      Observe that a representation $W$ of $G$ is a direct sum of faithful irreducible representations if and only if $W^H=0$ for every non-trivial normal subgroup $H\lhd G$. Thus, $V\new$ is the largest $G$-subrepresentation $W$ of $V$ such that $W^H = 0$ for all non-trivial normal subgroups $H\lhd G$.
      
      The kernel of the projection $V \to V\new$ is spanned by the subrepresentations $W$ of $V$ such that $W^H = W$ 
      for some non-trivial $H\lhd G$. In other words, the kernel is spanned by the subspaces $E(L) \otimes_\Z \Q$, where $L$ is a proper Galois subextension of $K/F$. It follows that $E(K/F)\gnew \otimes_\Z \Q = V\new$. 
\end{proof}

\begin{remark}\label{rem:arbitrary-extension-rank-equivalence}
    For any field extension $L/\Q$, define $V_L = E(K)\tensor_\Z L$ and define $V_L\new$ in the analogous way. Then the same argument shows that $E(K/F)\gnew\tensor_\Z L = V_L\new$. 
\end{remark}

\begin{remark}\label{rem:Galois-new}
    Let $W= \Ind^{G} 1$ be the regular representation. As a representation of $G$, we can write $W = \bigoplus_\rho W_\rho$, where the sum is over the irreducible rational representations $\rho$ of $G$. As before, let $W\new = \bigoplus_{\eta}W_{\eta}$ be the subrepresentation of $W$, where the sum is over all faithful irreducible rational representations $\eta$ of $G$.  Let $V_{\l}(E) = T_\l(E)\tensor_{\Zl}\Ql$ be the $\l$-adic Galois representation attached to $E$. Then the (equivariant) Birch and Swinnerton-Dyer conjecture predicts the equality 
    \[\rk E(K/F)\gnew \stackrel{?}{=} \ord_{s=1} L(V_\l(E)\tensor W\new, s),\]
    where $V_\l(E)\tensor W\new$ is viewed as a representation of $\Ga F$. Note that this $L$-function is defined over $F$. On the other hand, it does not seem that $ E(K/F)\new$, as defined in the introduction, should correspond to an $L$-function over $F$ in general.
\end{remark}

\subsubsection{Non-Galois extensions}
Let $K/F$ be an arbitrary finite extension of number fields.  
 For any extension $L/F$ contained in $K$ we write $\widetilde{L}$ for its Galois closure over $F$. We let $N = \widetilde{K}$ be the Galois closure of $K$ itself.
\begin{definition}
    Let $E(K/F)\gnew = E(K)/\sum_LE(L)$, where the sum is over the subfields $F\sub L\subsetneq K$, such that the Galois closure of $L$ is a proper subfield of $N$.
\end{definition}

Let $V(N) = E(N)\tensor_\Z\Q$, and define $V(N)\new = \bigoplus_{\eta}V_\eta$, as in the previous section, where $\eta$ runs over the faithful irreducible $\Q$-representations of $\Gal(N/F)$. By \Cref{prop:new-rank-equivalence}, we have a short exact sequence of $\Gal(N/F)$-representations
\begin{equation}\label{eq:v-new-exact}
0\to V(N)\old\to V(N)\to V(N)\new\to 0,    
\end{equation}
where 
\begin{equation}\label{eq:v-old}
    V(N)\old = \sum_{\stackrel{F\sub L\subsetneq N}{L/F\text{ Galois}}}E(L)\tensor_\Z\Q.
\end{equation}

The $G$-new points of $E$ for $K/F$ are compatible with the $G$-new points of $E$ for $N/F$ in the following sense. 

\begin{proposition}\label{prop:new-rank-equivalence-non-galois}
    We have $E(K/F)\gnew\tensor_\Z\Q = (V(N)\new)^{\Gal(N/K)}$.  
\end{proposition}

\begin{proof}
    From ($\ref{eq:v-new-exact}$), we have a short exact sequence of vector spaces
    \[0\to  (V(N)\old)^{\Gal(N/K)}\to V(N)^{\Gal(N/K)}\to (V(N)\new)^{\Gal(N/K)}\to 0.\]
    Since 
    \[E(K/F)\gnew\tensor_\Z\Q = \dfrac{E(K)\tensor_\Z\Q}{\dsp\sum_{\stackrel{F\sub L\subsetneq K}{\widetilde{L} \neq N}}E(L)\tensor_\Z\Q},\] 
    it remains to show that \[(V(N)\old)^{\Gal(N/K)}=\br{\sum_{\stackrel{F\sub L\subsetneq N}{L/F\text{ Galois}}}E(L)\tensor_\Z\Q}^{\Gal(N/K)}= \sum_{\stackrel{F\sub L\subsetneq K}{\widetilde{L} \neq N}}E(L)\tensor_\Z\Q.\]
    We have
    \[(V(N)\old)^{\Gal(N/K)} =\br{\bigoplus_\mu V_\mu}^{\Gal(N/K)}= \bigoplus_\mu V_\mu^{\Gal(N/K)},\]
    where the sum is over all $\Q$-representations of $\Gal(N/F)$ that are not faithful.
     
    Let $\mu$ be a non-faithful representation of $\Gal(N/F)$ with kernel $\Gal(N/L)$ for some non-trivial Galois extension $L/F$. Then $L\cap K$ is a subfield of $K$ whose Galois closure is a proper subfield of $N$. By definition, $V_\mu^{\Gal(N/L)} = V_\mu$. Hence, 
    \[V_\mu^{\Gal(N/K)} = V_\mu^{\Gal(N/K)\cdot\Gal(N/L)} = V_\mu^{\Gal(N/L\cap K)} \subset E(L\cap K)\tensor_\Z \Q.\]
    It follows that
    \[(V(N)\old)^{\Gal(N/K)}  = \bigoplus_\mu V_\mu^{\Gal(N/K)}\sub \sum_{\stackrel{F\sub L\subsetneq K}{\widetilde{L} \neq N}}E(L)\tensor_\Z\Q.\]
     
    Conversely, if $F\sub L\subsetneq N$, then 
    \[E(L)^{\Gal(N/K)} = E(L)^{\Gal(N/K)\cdot \Gal(N/L)} = E(L)^{\Gal(N/L\cap K)} = E(L\cap K).\] 
    Moreover, if $L/F$ is Galois, then the Galois closure of $L\cap K$ is not $N$. Hence, by ($\ref{eq:v-old}$), we see that
    \begin{align*}
    (V(N)\old)^{\Gal(N/K)}&= \br{\sum_{\stackrel{F\sub L\subsetneq N}{L/F\text{ Galois}}}E(L)\tensor_\Z\Q}^{\Gal(N/K)}\\&\supset\sum_{\stackrel{F\sub L\subsetneq N}{L/F\text{ Galois}}}E(L\cap K)\tensor_\Z\Q\\&= \sum_{\stackrel{F\sub L\subsetneq K}{\widetilde{L} \neq N}}E(L)\tensor_\Z\Q,    
    \end{align*}
which finishes the proof.
\end{proof}

\subsection{Abelian varieties with $\zeta$-multiplication}\label{subsec:zetamult}

We recall the definition of an abelian variety over $F$ with $\zeta$-multiplication and its twists; for further details, see \cite{ShnidmanWeiss}*{\S 2}. Let $A/F$ be an abelian variety, and let $\zeta = \zeta_n \in \overline F^\times$ be a primitive $n$-th root of unity, where $n$ is an odd prime power.

\begin{definition}
 An abelian variety $A/F$ has $\zeta_n$-multiplication if there is a $G_F$-equivariant injective ring homomorphism $\Z[\zeta_n]\hookrightarrow\End_{\bar F} A$.
\end{definition}

If $d\in F\t$, then we define $A_d$ to be twist of $A$ corresponding to the cocycle that is the image of $d$ under 
\[F\t\to F\t/F^{\times 2n}\simeq H^1(F, \mu_{2n}) \to H^1(F, \Aut_{\bar F}(A)).\]

If $A$ has $\zeta$-multiplication, then the endomorphism $1-\zeta\: A_{F(\zeta)}\to A_{F(\zeta)}$ over $F(\zeta)$ descends to an isogeny defined over $F$, which we denote by  $\pi \colon A \to A^{(1)}$ \cite{ShnidmanWeiss}*{\S 2}. The abelian variety $A^{(1)}$ is the twist of $A$ corresponding to the cocycle $\sigma \mapsto \frac{1-\zeta^\sigma}{1-\zeta}\in H^1(F, \Z[\zeta]\t)$ \cite{ShnidmanWeiss}*{Lem.\ 2.3}. In particular, if $\zeta = \zeta_3$, $A^{(1)} \cong A_{-27}$, but, in general, $A^{(1)}$ is not isomorphic to $A_d$ for any $d$. 

The abelian variety $A^{(1)}$ also has $\zeta$-multiplication, and we let $\pi_d\:A_d\to A^{(1)}_d$ denote the twist of $\pi$ corresponding to $d\in F\t/F^{\times 2n}$. 

\subsection{An auxiliary abelian variety}\label{subsec:aux}

Let $n = 3^m$ for some $m\ge 1$ and let $\zeta = \zeta_n$. Fix an element $d\in F\t/F^{\times 2n}$ of order $2n$, and let $K_d = F(\!\sqrt[2n]d)$. We  will define an abelian variety $B/F$ with $\zeta$-multiplication, such that $\rk B_d(F) = \rk E(K_d/F)\new$.

View $\Z[\zeta]$ as a right $G_F$-module.  Following \cite{MRS}*{Thm.\ 1.8}, we define the abelian variety 
\[B:= \Z[\zeta] \otimes_\Z E\]
over $F$, which has dimension $2\cdot 3^{m-1}$. By \cite{MRS}*{Cor.\ 1.7},  there are $G_F$-equivariant ring embeddings $\Z[\zeta] \hookrightarrow \End_\Z(\Z[\zeta])\hookrightarrow \End_{\bar F}(B)$, so that $B$ has $\zeta$-multiplication. Concretely, the $\zeta$-multiplication on $B$ is by left multiplication on the $\Z[\zeta]$-factor.

\begin{example}\label{ex:Bform=1}
Suppose that $m = 1$ and $F = \Q$. Then $\Z[\zeta]$ is a rank two $\Z$-module and  $B$ is an abelian surface. The embedding $\tau \colon E \hookrightarrow B$ defined by $P \mapsto 1 \otimes P$, has cokernel $B/\tau(E)$ isomorphic to $(\Z[\zeta]/\Z) \otimes E$. The latter is isomorphic to the quadratic twist $E_{-3}$, since $\Z[\zeta]/\Z$ is free of rank one with Galois action factoring through $\Gal(\Q(\zeta)/\Q) = \Gal(\Q(\sqrt{-3})/\Q)$. On the other hand, we can embed $E_{-3}$ in $B$ via $P \mapsto \sqrt{-3} \otimes P$. If $E$ and $E_{-3}$ are not isogenous, then both $\Hom_\Q(E,B)$ and $\Hom_\Q(E_{-3}, B)$ must have rank 1. They are visibly generated by the two inclusions already mentioned. Since the composition $E_{-3} \hookrightarrow B \to E_{-3}$ is multiplication by 2, the map $B \to E_{-3}$ has no section, and the intersection of $E$ and $E_{-3}$ in $B$ is via the isomorphism $\eta \colon E[2] \simeq E_{-3}[2]$. The upshot is that $B \simeq (E \times E_{-3})/\Delta$, where $\Delta$ is the graph of $\eta$. By specialization, this is true even if $E$ is isogenous to $E_{-3}$.
\end{example}

\begin{lemma}\label{lem:twists-isomorphic}
    If $i\in (\Z/2n\Z)\t$, then $B_{d}\simeq B_{d^i}$ for any $d\in F\t/F^{\times 2n}$.
\end{lemma}

\begin{proof}
    Let $\I_d = \Z[\zeta]$ with the following twisted Galois action: if $\sigma\in G_F$ and $x\in \I_d$, we define \[x^{\sigma_d} = \frac{\sqrt[2n]d}{\sqrt[2n]d^\sigma}x^\sigma,\]
    where $x^\sigma$ is the usual Galois action on $\Z[\zeta]$. Viewing $\I_d$ as a $\Z$-module with a $G_F$ action, we can define the abelian variety $\I_d \tensor_\Z E$. 
    
    We show that this abelian variety is isomorphic to $B_d$. Over $F(\!\sqrt[2n]d)$, the isomorphism $\I_d\to \I_1 = \Z[\zeta]$ defines an isomorphism $\psi\: \I_d \tensor_\Z E \to B$ by \cite{MRS}*{Cor.\ 1.9}. The cocycle $\psi^\sigma\psi\ii\in \Aut(B)$ is the map $\sigma\mapsto \frac{\sqrt[2n]d^\sigma}{\sqrt[2n]d}$. Indeed, if $x\tensor P\in B$, then
    \begin{align*}
    \psi^\sigma\psi\ii(x\tensor P) &= \sigma\circ\psi\circ\sigma_d\ii\circ\psi\ii (x\tensor P)\\ &=\sigma\circ\psi (\frac{\sqrt[2n]d}{\sqrt[2n]d^{\sigma\ii}}x^{\sigma\ii}\tensor P^{\sigma\ii})\\
    &= \frac{\sqrt[2n]d^\sigma}{\sqrt[2n]d}(x\tensor P).    
    \end{align*}
By Kummer theory, this is precisely the cocycle which classifies $B_d$. Hence, $\I_d \tensor_\Z E\simeq B_d$.

Now, as $\Z[G_F]$-modules, there is an isomorphism $\I_d\simeq\I_{d^i}$ given by $-\zeta\mapsto(-\zeta)^{i}$. Hence, by \cite{MRS}*{Cor.\ 1.9}, $B_d\simeq B_{d^i}$.
\end{proof}

\begin{proposition}\label{lem:new-rank-comparison}
    Let $d\in F\t/F^{\times 2n}$. If $d$ has order $2n$ as an element of $F(\zeta)\t/F(\zeta)^{\times 2n}$, then $\rk E(K_d/F)\new = \rk B_d(F)$.
\end{proposition}

\begin{proof}
    First suppose that $\zeta\in F$, so that $K_d/F$ is Galois, and let $G = \Gal(K_d/F)$. For this extension, we have  $E(K_d/F)\new  = E(K_d/F)\gnew$,
    so we can use the representation-theoretic interpretation of \Cref{sec:new-rank}. 
    
    By the assumption that $d$ has order $2n$, we have  $G \simeq C_{2n}$. Let $\epsilon\:G\to\C\t$ be the character 
    $\sigma\mapsto \sqrt[2n]d^{\sigma - 1}$, 
    so that the set of faithful irreducible $\C$-valued representations of $G$ is $\{\epsilon^i : i\in (\Z/2n\Z)\t\}$. By \Cref{prop:new-rank-equivalence} and the subsequent remark, we have
    \[E(K_d/F)\new\tensor_\Z\C = \bigoplus_{i\in (\Z/2n\Z)\t} V_i,\]
    where
    \[V_i := \set{v\in E(K_d)\tensor_\Z\C : \sigma(v) = \epsilon^i(\sigma)v,\ \forall\sigma\in G}.\]
    
    On the other hand, we can view $B(K_d)$ as a finitely generated $\Z[\zeta][G]$-module. As in the proof of \cite{MRS}*{Thm.\ 2.2}, we have $B(K_d) = (E\tensor_\Z \Z[\zeta])(K_d) = E(K_d)\tensor_\Z\Z[\zeta]$. Hence, as $\C$-vector spaces, we have $B(K_d)\tensor_{\Z[\zeta]}\C = E(K_d)\tensor_\Z\C$. 
    
    Thus, we may view $V_i$ as a subspace of $B(K_d)\tensor_\Z\C$. But since $B_{d^i}$ is the twist of $B$ corresponding to the cocycle $\epsilon^i$, we have 
    \[V_i = B_{d^i}(F)\tensor_{\Z[\zeta]}\C,\]
    and hence
    \[\dim_{\C}V_i = \rk_{\Z[\zeta]}B_{d^i}(F) = \frac{1}{2\cdot 3^{m-1}}\rk B_{d^i}(F).\]
    
    It follows that
    \[\rk E(K_d/F)\new = \frac{1}{2\cdot 3^{m-1}}\sum_{i\in (\Z/2n\Z)\t}\rk B_{d^i}(F) = \rk B_d(F),\]
    where the final equality follows from \Cref{lem:twists-isomorphic}.
    
    Now suppose that $\zeta\notin F$, let $K_d = F(\!\sqrt[2n]d)$, and let $N_d = K_d(\zeta)$ be its Galois closure. Then we have $E(N_d/F)\gnew = E(N_d/F(\zeta))\gnew$. Indeed, by assumption, $d$ has order $2n$ in $F(\zeta)\t/F(\zeta)^{\times 2n}$. Hence, if $L/F$ is a Galois extension with $L\subsetneq N_d$, then $L(\zeta)/F(\zeta)$ is a Galois extension and $L(\zeta)\subsetneq N_d$.
  
    Therefore, as before, we have
    \[E(N_d/F)\gnew\tensor\C =E(N_d/F(\zeta))\gnew\tensor\C = \bigoplus_{i\in (\Z/2n\Z)\t} B_{d^i}(F(\zeta))\tensor_{\Z[\zeta]}\C.\]
    By \Cref{prop:new-rank-equivalence-non-galois}, we have
    \[E(K_d/F)\gnew\tensor\C = (E(N_d/F)\gnew \tensor\C)^{\Gal(N_d/K_d)} = \bigoplus_{i\in (\Z/2n\Z)\t} B_{d^i}(F)\tensor_{\Z[\zeta]}\C.\]
  It follows that
    \[\rk E(K_d/F)\new = \rk E(K_d/F)\gnew = \frac{1}{2\cdot 3^{m-1}}\sum_{i\in (\Z/2n\Z)\t}\rk B_{d^i}(F) = \rk B_d(F),\]
    where the final equality follows from \Cref{lem:twists-isomorphic}.  
\end{proof}

\section{The average rank of $B_d$}\label{sec:avg-rank}

In this section we prove \Cref{thm:rankgainintro}. Let $E$ be as in the theorem and let $B$ be the abelian variety defined in Section \ref{subsec:aux}. By \Cref{lem:new-rank-comparison}, to prove \Cref{thm:rankgainintro}, it is enough to prove that the average rank of $B_d(F)$ is bounded. 

\subsection{Average ranks in cyclotomic twist families}

We recall a general average rank result for a twist family of an abelian variety $A/F$ with $\zeta$-multiplication (see \Cref{subsec:zetamult}).  

There is a natural height function on $F\t/F^{\times 2n}$, which we now define. Let $M_\infty$ be the set of archimedean places of $F$. For each $d \in F\t/F^{\times 2n}$ we choose a lift $d_0 \in F^\times$, and then define the ideal $I = \{a \in F \colon a^{2n} d_0 \in \O_F\}.$

\begin{definition}\label{def:height}
The height of $d$ is  $H(d) = \Nm(I)^{2n} \prod_{v \in M_\infty} |d_0|_v.$
\end{definition}

This definition is independent of the lift $d_0$, by the product formula. If $F = \Q$, then $H(d) = |d_0|$, where $d_0$ is the unique $2n$-th power free integer representing $d$. For any $X > 0$, the set  \[\Sigma_X = \{d \in F\t/F^{\times 2n} \colon H(d) < X\}\] 
is finite. Thus, we can define the average rank of $A_d(F)$ to be
\[\avg_d \rk A_d(F) = \lim_{X \to \infty} \underset{d \in \Sigma_X}\avg \rk A_d(F). \]
If the limsup is finite, we say that the \emph{average rank of $A_d(F)$ is bounded}.

Recall from \Cref{subsec:zetamult} that there exists an isogeny $\pi \colon A \to A^{(1)}$, which is a descent of $1-\zeta$ to $F$. Let $A[\pi]$ denote the kernel of this isogeny.

\begin{theorem}[\cite{ShnidmanWeiss}*{Thm.\ 1.1}]\label{thm:completely-reducible}
Let $A$ be an abelian variety with $\zeta_{3^m}$-multiplication over $F$. If the $G_F$-module $A[\pi]$ is a direct sum of characters, then $\avg_d \rk A_d(F)$ is bounded.
\end{theorem}

\subsection{An isogenous abelian variety}

To apply \Cref{thm:completely-reducible} to the abelian variety $B$, we would need to know that the $G_F$-representation $B[\pi]$ is a direct sum of characters. However, the following lemma shows that this is not the case in general.

\begin{lemma}\label{lem:pitorsion}
    There is an isomorphism of $\F_3[G_F]$-modules $B[\pi]\xrightarrow{\sim} E_{-3}[3]$.
\end{lemma}
\begin{proof}
By \cite{MRS}*{Thm.\ 2.2}, there is a $G_F$-equivariant isomorphism 
\[B[3] \simeq \Z[\zeta] \otimes_\Z E[3] \simeq \Z[\zeta]/3\Z[\zeta] \otimes_\Z E[3].\] Let $\p = (1-\zeta)\Z[\zeta]$, so that $\p^r = 3\Z[\zeta]$, for $r = 2\cdot 3^{m-1}$.  Then 
\[B[\pi] \simeq \p^{r-1}/\p^r \otimes_\Z E[3] \simeq \p^{-1}/\Z[\zeta] \otimes_\Z E[3].\]  
Let $\chi\colon G_F\to \Z_3\t\hookrightarrow \Z_3[\zeta]\t$ denote the $3$-adic cyclotomic character, and let $\overline\chi \equiv \chi\pmod{1-\zeta}$ be the mod $3$ cyclotomic character.
 Using the identity
\[\br{\frac{1-\zeta^{\chi(\sigma)}}{1-\zeta}} = 1 +\zeta + \cdots + \zeta^{\chi(\sigma)-1} \equiv\chi(\sigma)\pmod{(1-\zeta)},\]
we see that the $G_F$-action on the one-dimensional $\F_3$-vector space $\p^{-1}/\Z[\zeta]$ is by $\overline\chi$. Thus, 
\[B[\pi] \simeq \overline\chi \otimes_\Z E[3] \simeq (\overline\chi \otimes E)[3] \simeq E_{-3}[3].\] 
Explicitly, if $P \in E[3](\overline{F})$, then $\frac{3}{1-\zeta} \otimes P \in B[\pi](\overline{F})$.  
\end{proof}

Hence, to apply \Cref{thm:completely-reducible}, we instead consider an abelian variety $A$ that is isogenous to $B$, for which $A[\pi]$ is completely reducible. Recall that $E$ admits a $3$-isogeny $\theta\: E\to E'$. Let $\tau \colon  E_{-3}[\theta_{-3}]\hookrightarrow B[\pi]$ be the embedding induced by \Cref{lem:pitorsion}. 

\begin{definition}\label{def:A}
    Define $A = B/\tau(E_{-3}[\theta_{-3}])$.
\end{definition}

By \cite{ShnidmanWeiss}*{Lem.\ 2.6}, the fact that $E_{-3}[\theta_{-3}]\sub B[\pi]$ ensures that $A$ also has $\zeta$-multiplication. We use a slight abuse of notation and write $\pi = \pi_A \colon A \to A^{(1)}$ for the descent of $1-\zeta \in \End_{\bar F}(A)$ to $F$. 

\begin{lemma}\label{lem:completely-reducible}
    There is an isomorphism $A[\pi]\simeq  E_{-3}'[\widehat\theta_{-3}]\times E[\theta]$ of $\F_3[G_F]$-modules. 
\end{lemma}

\begin{proof}
    We can identify  $E'_{-3}[\widehat\theta_{-3}]\simeq E_{-3}[3]/E_{-3}[\theta_{-3}] \subset A$. Its image in $A[\pi]$ is represented by elements of the form $\frac{3}{1-\zeta} \otimes P$.  On the other hand, we can embed $E[\theta] \hookrightarrow A$ by sending $Q$ to the image of $\frac{3}{(1-\zeta)^2} \otimes Q$ in $A$. This element is evidently killed by $\pi$, since $\frac{3}{1-\zeta} \otimes E[\theta]$ is 0 in $A$.  The image of this copy of $E[\theta]$ is also visibly independent of $E'_{-3}[\widehat\theta_{-3}]$.
 \end{proof}

\begin{proof}[Proof of Theorem $\ref{thm:rankgainintro}$]
    By \Cref{lem:completely-reducible}, the finite $G_F$-module $A[\pi]$ is a direct sum of characters. Hence, by \Cref{thm:completely-reducible}, the average $\avg_d\rk A_d(F)$ is bounded. Since $A_d$ and $B_d$ are isogenous, it follows that $\avg_d\rk B_d(F)$ is bounded. Now, all but finitely many $d\in F\t/F^{\times 2n}$ of order $2n$ have order $2n$ in $F(\zeta)$ as well. Indeed, the kernel of the map $F\t/F^{\times 2n}\to F(\zeta)\t/F(\zeta)^{\times 2n}$ is isomorphic to $H^1(\Gal(F(\zeta)/F), \mu_{2n})$, which is a finite group. For these $d$, we have $\rk B_d(F) = \rk E(K_d/F)\new$, by \Cref{lem:new-rank-comparison}, so the average rank of $E(K_d/F)\new$ is indeed bounded.
\end{proof}

\section{Intersecting Selmer groups and the correlation trick}\label{sec:pos-prop}

We consider a general situation involving twist families of abelian surfaces with $\zeta_3$-multiplication.  The main result of the section is \Cref{prop:ab-surface-pos-prop}, which will be used in the proof of \Cref{thm:pos-prop}.  

\subsection{Set-up}

Let $F$ be a number field, and let $\zeta = \zeta_3$ be a primitive cube root of unity in $\overline F^\times$. Let $A/F$ be an abelian surface with $\zeta$-multiplication, as in Section \ref{subsec:zetamult}. We assume that $A$ admits a polarization $\lambda \colon A \to \widehat A$ whose degree is not divisible by $3$ and such that the Rosati involution $\alpha \mapsto \lambda^{-1} \widehat\alpha \lambda$ on $\End(A)$ restricts to complex conjugation on the subring $\Z[\zeta]$. We also assume that $A[\pi](F) = \langle P, Q\rangle$, or in other words, that $A[\pi]$ has trivial $G_F$-action.

For each $d\in F^\times$, recall that $A_d$ is the twist of $A$ corresponding to the cocycle $\xi_\sigma =  \sqrt[6]d^{\sigma-1}$ in $H^1(F, \mu_{6})$. The isomorphism $\Aut(A)\to\Aut(\widehat{A})$ sending $f\mapsto\widehat{f}\ii$ induces an isomorphism $H^1(F, \Aut(A))\simeq H^1(F, \Aut(\widehat{A}))$, which we will denote by $\xi \mapsto \widehat{\xi}\ii$.

\begin{definition}
    Define $(\widehat{A})_d$ to be the twist of $\widehat{A}$ corresponding to the cocycle $\widehat{\xi}\ii$.
\end{definition}

\begin{lemma}
    Let $\widehat A_d$ be the dual of $A_d$. Then $(\widehat{A})_d\simeq\widehat{A}_d$.
\end{lemma}

\begin{proof}
 Let $f\:A\to A_d$ be an isomorphism over $K$. Then, by definition, we have $\xi_\sigma = f\ii f^\sigma$. By duality, we have an isomorphism $\widehat f\:\widehat{A}_d\to \widehat{A}$. Consider the cocycle $\widehat{f}(\widehat{f\ii})^\sigma$ corresponding to the twist $\widehat{A}_d$. Then
\[\widehat{f}(\widehat{f\ii})^\sigma = \widehat{(f\ii)^\sigma f} = (\widehat{f\ii f^{\sigma\ii})^\sigma} = \widehat{\xi_{\sigma\ii}^\sigma} = \widehat{\xi_\sigma\ii}.\]
Here, the final equality comes from the cocyle relation $1 = \xi_1 = \xi_{\sigma\ii}^\sigma\xi_\sigma $. 
Thus, the cocycle corresponding to $\widehat{A}_d$ is $\widehat\xi\ii$, the same cocycle used to define $(\widehat A)_{d}$. It follows that $\widehat{A}_d\simeq (\widehat A)_{d}$.
\end{proof}

For each $d\in F^\times$, let $\pi_d \colon A_d \to A_{-27d}$ be the $d$-th twist of $\pi \colon A \to A^{(1)} = A_{-27}$. The endomorphism $[3]\:A_d\to A_d$ factors as $u\circ\pi_{-27d}\circ\pi_d$, for some automorphism $u$.
Let $\widehat\pi_d \colon \widehat{A}_{-27d}\to \widehat A_d$ denote the isogeny dual to $\pi_d$.

\begin{lemma}\label{lem:equal-selmer}
 We have  $\Sel(\pi_{-27d}) \simeq \Sel(\widehat\pi_d)$.
\end{lemma}
\begin{proof}
Let $K = F(\zeta, \sqrt[6]d)$ and let $f\: A_K\to (A_d)_K$ be an isomorphism. By \cite[Prop.\ 2.2]{Howe-isogenyclasses} and the condition on the Rosati involution, the polarization $\widehat{f\ii}\lambda_Kf\ii$ of $(A_d)_K$ descends to a polarization of $A_d$ over $F$, which we will denote by $\lambda_d\: A_d\to \widehat A_d$.  Consider the diagram
\[\begin{tikzcd}
A_d \arrow[r, "\lambda_d"] \arrow[d, "\pi_d"'] & \widehat A_d \arrow[d, "\widehat\pi_{-27d}"] \\
A_{-27d} \arrow[r, "\lambda_{-27d}"]           & \widehat A_{-27d}                           
\end{tikzcd}\]

The definition of $\lambda_d$ and the condition on the Rosati involution shows that this diagram commutes. In particular, $\Sel(\lambda_{-27d}\circ\pi_d) \simeq \Sel(\widehat\pi_{-27d}\circ\lambda_d)$, and since $\lambda_d$ is prime-to-$3$, it follows that $\Sel(\pi_d)\simeq \Sel(\widehat\pi_{-27d})$.
\end{proof}

By the Lemma and the equality $[3] = u \circ \pi_{-27} \circ \pi_d$, to control $\Sel_3(A_d)$ it is enough to control the two Selmer groups $\Sel(\pi_d)$ and $\Sel(\widehat{\pi}_d)$. 

\subsection{Selmer ratios}

Let $\alpha \colon X \to Y$ be an isogeny of abelian varieties over $F$. The global Selmer ratio of $\alpha$ is by definition the product $c(\alpha) = \prod_v c_v(\alpha)$ of local Selmer ratios 
\[c_v(\alpha) = \dfrac{\#\mathrm{coker}\left(X(F_v) \to Y(F_v)\right)}{\#\mathrm{ker}\left(X(F_v) \to Y(F_v)\right)},\]
one for each place $v$ of $F$.
 If $\deg(\alpha)$ is a power of $3$, then for $v \nmid 3\infty$, we have $c_v(\alpha) = c_v(Y)/c_v(X)$, where $c_v(X)$ is the Tamagawa number of $X$ over $F_v$ \cite{schaefer}*{Lem.\ 3.8}. Thus, up to some subtle factors at places $v$ above $3$ and $\infty$, the number $c(\alpha)$ is the ratio of the global Tamagawa numbers $c(Y)/c(X)$. In particular, we have $c_v(\alpha) \in 3^\Z$, and $c_v(\alpha) = 1$ for all but finitely many $v$.

Let $\widehat{\alpha}\:\widehat{Y}\to \widehat{X}$ be the dual isogeny. Then the Greenberg-Wiles formula \cite[8.7.9]{NSW:cohomologyofnumberfields} reads
    \begin{equation}\label{eq:GreenbergWiles}
    c(\alpha) =  \dfrac{\#\Sel(\alpha)}{\#\Sel(\widehat\alpha)} \cdot \dfrac{\#\widehat Y[\widehat\alpha](F)}{\# X[\alpha](F)}.
    \end{equation}

\subsection{The correlation trick}

By \Cref{lem:equal-selmer}, we have 
\begin{equation}\label{eq:bounds}
\rk A_d(F) \le \dim_{\F_3}\Sel_3(A_d) \le \dim_{\F_3}(\Sel(\pi_d) \+\Sel(\pi_{-27d})) = \dim_{\F_3}(\Sel(\pi_d) \+ \Sel(\widehat\pi_d)).
\end{equation}
Now, since $A[\pi](F) = \langle P, Q\rangle$, we can factor $\pi_d$ in multiple ways as a chain of $3$-isogenies:
\[\begin{tikzcd}
                                                                               & \left(\frac{A}{\langle P\rangle}\right)_d \arrow[rd, "\phi'_{d}"]            &          \\
A_d \arrow[ru, "\phi_d"] \arrow[rd, "\psi_d"'] \arrow[r, "\eta_d" description] & \left(\frac{A}{\langle P+Q\rangle}\right)_d \arrow[r, "\eta'_d" description] & A_{-27d} \\
                                                                               & \left(\frac{A}{\langle Q\rangle}\right)_d \arrow[ru, "\psi'_d"']             &         
\end{tikzcd}\]

and by duality, we obtain a corresponding factorisation for $\widehat\pi$:
\[\begin{tikzcd}
                                                                                                                              & \widehat{\left(\frac{A}{\langle P\rangle}\right)}_{d} \arrow[rd, "\widehat\phi_d"]              &              \\
\widehat A_{-27d} \arrow[ru, "\widehat\phi'_{d}"] \arrow[rd, "\widehat\psi'_{d}"'] \arrow[r, "\widehat\eta'_{d}" description] & \widehat{\left(\frac{A}{\langle P+Q\rangle}\right)}_{d} \arrow[r, "\widehat\eta_d" description] & \widehat A_d \\
                                                                                                                              & \widehat{\left(\frac{A}{\langle Q\rangle}\right)}_{d} \arrow[ru, "\widehat\psi_d"']             &             
\end{tikzcd}\]

The following theorem is the main result of this section:
\begin{theorem}\label{prop:ab-surface-pos-prop}
    Let $A/F$ be as above. Suppose that the set 
    \[T = \{d\in F^\times/F^{\times 6}: c(\phi_d) = c(\phi'_d) = c(\psi_d) = c(\psi'_d) = 1\}\]
    has positive density. Then for a positive lower density of $d\in F^\times/F^{\times 6}$, we have $\rk A_d(F) = 0$.
\end{theorem}

\begin{lemma}[\cite{ShnidmanWeiss}*{Lem.\ 6.2}]\label{lem:injection-of-selmer}
  Let $F$ be any field, and suppose that there is a commutative diagram of isogenies of abelian varieties over $F$,
	\begin{center}
		\begin{tikzcd}
A \arrow[r, "\phi_1"] \arrow[d,"\phi_2"] & B_1 \arrow[d, "\psi_2"] \\
		B_2 \arrow[r, "\psi_1"]       & {C}
		\end{tikzcd}\end{center}
	such that $\phi_2$ maps $A[\phi_1]$ isomorphically onto $B_2[\psi_1]$. Then $\psi_2$ induces an injection
	\[\frac{B_1(F)}{\phi_1(A(F))}\hookrightarrow\frac{C(F)}{\psi_1(B_2(F))}.\]
	As a consequence, if $F$ is a number field, then the map $\phi_2$ induces an embedding $\Sel(\phi_1)\hookrightarrow\Sel(\psi_1)$.
\end{lemma}

\begin{corollary}\label{cor:identify-maps}
    For almost all $d\in T$, we have $\Sel(\phi_d) = \Sel(\psi'_d)$ and $\Sel(\phi'_{d})= \Sel(\psi_d)$.
\end{corollary}

\begin{proof}
    By \Cref{lem:injection-of-selmer}, we have $\Sel(\phi_d)\hookrightarrow\Sel(\psi'_d)$ and $\Sel(\widehat\psi'_{d})\hookrightarrow\Sel(\widehat\phi_d)$. For almost all $d\in T$, the abelian varieties in the diagram have no non-trivial rational $3$-torsion points, so by $(\ref{eq:GreenbergWiles})$ we have $\#\Sel(\phi_d) = \#\Sel(\widehat\phi_d)$ and $\#\Sel(\psi'_d) = \#\Sel(\widehat\psi'_d)$. Thus \[\#\Sel(\phi_d)\leq \#\Sel(\psi'_d) = \#\Sel(\widehat\psi'_d)\leq\#\Sel(\widehat\phi_d)= \#\Sel(\phi_d).\]    It follows that $\Sel(\phi_d) = \Sel(\psi'_d)$. The proof of the second equality is identical. 
\end{proof}

\begin{proof}[Proof of Theorem $\ref{prop:ab-surface-pos-prop}$]
    By \cite{ShnidmanWeiss}*{Thm.\ 5.2}, we have $\avg_{d\in T}\#\Sel(\phi_d) = \avg_{d\in T}\#\Sel(\phi'_d) = 2$, and by \cite{ShnidmanWeiss}*{Thm.\ 5.3}, $\avg_{d\in T}\#\Sel(\eta_d) >1$.  Let $\mathrm{min}_d = \min(\#\Sel(\phi_d),\#\Sel(\phi'_d))$ and $\mathrm{max}_d = \max(\#\Sel(\phi_d),\#\Sel(\phi'_d))$. By \Cref{lem:injection-of-selmer} and the above diagram, we have $\#\Sel(\eta_d)\le \#\Sel(\phi'_d)$ and $\#\Sel(\eta_d)\le \#\Sel(\psi'_d)$ for almost all $d\in T$. Therefore, by \Cref{cor:identify-maps}, we have $\#\Sel(\eta_d) \leq \mathrm{min}_d$ for almost all $d\in T$. Hence $\liminf\avg_{d\in T} \mathrm{min}_d >1$, where
    \[\liminf\avg_{d\in T}\mathrm{min}_d := \liminf_{X\to \infty}\dfrac{\sum_{d\in T, |d|\le X} \mathrm{min}_d}{\sum_{d\in T, |D|\le X}1}.\]
    
    On the other hand, we have
    \begin{align*}
        4 &= \avg_{d\in T}\#\Sel(\phi_d) +\avg_{d\in T}\#\Sel(\phi'_d) \\&= \avg_{d\in T}(\mathrm{min}_d + \mathrm{max}_d)
        \\ &\ge \liminf\avg_{d\in T}\mathrm{min}_d+\liminf\avg_{d\in T}\mathrm{max}_d
        \\&>1+ \liminf\avg_{d\in T}\mathrm{max}_d,
    \end{align*}
    so that $\liminf\avg_{d\in T}\mathrm{max}_d < 3.$ This immediately implies that for a positive lower density of twists $d\in T$, we have 
    \[\#\Sel(\phi_d) = \#\Sel(\phi'_d) = 1.\] 
    More quantitatively, let $s_0$ be the proportion of $d\in T$ such that $\mathrm{max}_d = 1$. Then
     \begin{equation}\label{eq:pos-prop}
     s_0 + 3(1-s_0) <3    
     \end{equation}
     and hence $s_0>0$. 
    By the Greenberg--Wiles formula and the definition of $T$, we have $\#\Sel(\widehat\phi_d) = \#\Sel(\widehat\phi'_{d}) = 1$ for almost all such $d$ as well. By (\ref{eq:bounds}), we have $\#\Sel_3(A_d) = 1$ for all such $d$, and hence $\rk A_d(F) = 0$ as well.
    \end{proof}

\begin{remark}\label{rem:explicit-pos-prop}
    Suppose that for some $m\in\Z_{\ge 0}$, $c(\eta_d) \ge 3^{-m}$ for all $d\in T$. Then, by \cite{ShnidmanWeiss}*{Thm.\ 5.2}, we have $\avg_{d\in T}\#\Sel(\eta_d) \ge 1 + 3^{-m}$, and $(\ref{eq:pos-prop})$ becomes
    \[s_0 + 3(1-s_0) \leq 3-3^{-m}.\]
    It follows that $s_0 \geq \frac{1}{2\cdot 3^m}$. 
\end{remark}

\section{Proof of Theorem $\ref{thm:pos-prop}$}\label{sec:proof-pos-prop}

Recall that $E\: y^2 + axy +by = x^3$ is an elliptic curve over $\Q$ with a point $(0,0)$ of order $3$.  Let $\theta\:E\to E' = E/\langle (0, 0)\rangle$ be the corresponding $3$-isogeny. Let $\zeta=\zeta_3$ be a primitive cube root of unity. In \Cref{def:A}, we defined an abelian variety $A/\Q$ with $\zeta$-multiplication, such that for each $d\in \Q\t/\Q^{\times 6}$, the rank of $A_d(\Q)$ is equal to $\rk E(K_d/\Q)\new$ (see \Cref{lem:new-rank-comparison}). To prove \Cref{thm:pos-prop}, we will show that $\rk A_d(\Q) = 0$ for a positive proportion of twists.

By \Cref{lem:completely-reducible}, we have $A[\pi] \simeq E'_{-3}[\widehat\theta_{-3}]\times E[\theta]$. We have $E[\theta] = \langle (0, 0) \rangle$, and by duality, we have $E[\theta]\simeq E'_{-3}[\widehat\theta_{-3}]$ as $G_F$-modules. Let $P$ and $Q$ denote the images in $A[\pi]$ of generators of $E[\theta]$ and $E'_{-3}[\widehat\theta_{-3}]$. Then $A[\pi](\Q) = \langle P, Q\rangle$, and we are in the setting of \Cref{prop:ab-surface-pos-prop}. Before proceeding, we need to verify that $A$ admits a prime-to-3 polarization satisfying the assumptions of \Cref{prop:ab-surface-pos-prop}.

\begin{lemma}\label{lem:jacobian-curve}
    Let $C$ be the hyperelliptic curve $y^2 = x^6 + \alpha x^3 + 1$, where $\alpha = 108b/a^3-2$. Then $A\simeq\Jac(C)$ and the $\zeta$-multiplication on $A$ is induced from the order $3$ automorphism $(x,y) \mapsto (\zeta x, y)$ of $C$. 
\end{lemma}

\begin{proof}
    Set  $E^+ = E$ and $E^- = E'_{-3}$.  It is convenient to use the following symmetric models 
    \[E^\pm \colon y^2 = x^3 + (3x + 2\pm\alpha)^2.\] 
    The double covers $f_\pm \colon C \to E^\pm$ given by
    \[(x :  y : z)\mapsto \br{\frac{(\alpha \pm 2)xz}{(x\mp z)^2}, \frac{(\alpha\pm2)y}{(x\mp z)^3}}\]
    may be used to define a morphism $f = f^*_+ - f^*_- \colon  E^+ \times E^- \to \Pic^0(C) \simeq \Jac(C)$, by pullback of divisors. Note that $C$ has an involution $\tau \colon (x : y : z) \mapsto (z: y : x)$ and the two double covers above are the quotients by $\tau$ and $\iota\tau$, where $\iota$ is the hyperelliptic involution.  The kernel of $f$ is therefore the intersection $f^*_+E^+ \cap f^*_-E^-$, consisting of divisors fixed by both of these involutions. Thus, the kernel is precisely $f^*_+E^+[2] \cap f^*_-E^-[2]$. Hence, $\Jac(C)\simeq (E^+\times E^-)/\Gamma$,
    where $\Gamma$ is the graph of the isomorphism $E^+[2]\simeq E^-[2]$.
    
    On the other hand, the abelian surface $B = \Z[\zeta] \otimes E$ is isomorphic to $(E \times E_{-3})/\Delta$, where $\Delta$ is the graph of $E[2] \simeq E_{-3}[2]$, by \Cref{ex:Bform=1}. This gives the claimed isomorphism
    \[A \simeq B/E_{-3}[\theta_{-3}] \simeq (E^+ \times E^-)/\Gamma \simeq \Jac(C).\]
    To see that the $\zeta$-actions match up, it is enough to show that $A[1- \zeta]$ maps to $\Jac(C)[1 -\zeta]$ (for the respective automorphisms $\zeta$ on each) under the above isomorphism.  By \Cref{lem:completely-reducible}, we have $A[1-\zeta] \simeq E^+[\theta^+] \times E^-[\theta^-]$.  This is also $\Jac(C)[1 - \zeta]$, since the divisors fixed by $\zeta$ are generated by the difference of the four points on $C$ where $xz = 0$, and these visibly map to $\theta$-torsion points of $E^\pm$, since these are the points where $x = 0$.
\end{proof}

\begin{remark}
The genus two curves $y^2 = x^6 + \alpha x^3 + 1$ were considered in \cite{BruinFlynnTesta}, where it is remarked that these are precisely the genus two curves with two independent 3-torsion divisors supported on at most 4 points.
\end{remark}

\begin{corollary}\label{cor:principally-polarized}
    $A$ is principally polarized and its polarization is preserved by $\zeta$.
\end{corollary}
To say that the principal polarization $\lambda \colon A \to \widehat A$ is preserved by $\zeta$ is equivalent to the condition that the corresponding Rosati involution on $\End(A)$ restricts to complex conjugation on $\Z[\zeta]$.  Thus, $A$ satisfies the conditions of \Cref{prop:ab-surface-pos-prop}.

Recall the commutative diagram
\[\begin{tikzcd}
                                                                               & \frac{A}{\langle P\rangle} \arrow[rd, "\phi'"]            &          \\
A \arrow[ru, "\phi"] \arrow[rd, "\psi"'] \arrow[r, "\eta" description] & \frac{A}{\langle P+Q\rangle} \arrow[r, "\eta'" description] & A_{-27} \\
                                                                               & \frac{A}{\langle Q\rangle} \arrow[ru, "\psi'"']             &         
\end{tikzcd}\]
from \Cref{prop:ab-surface-pos-prop}. Up to $2$-isogenies, this diagram becomes
\[\begin{tikzcd}
                                                                                                                                  & E'_{-3}\times E' \arrow[rd, "\widehat\theta_{-3}\times 1"] &                 \\
E_{-3}'\times E \arrow[ru, "1\times \theta"] \arrow[rd, "\widehat\theta_{-3}\times 1"'] \arrow[r, "\eta" description] & \frac{A}{\langle P+Q\rangle} \arrow[r, "\eta'" description]     & E_{-3}\times E' \\
                                                                                                                                  & E_{-3}\times E \arrow[ru, "1\times \theta"']           &                
\end{tikzcd}\]

Let $\f$ be the conductor of $E$ and let $\Sigma$ be the set of integers $d$ such that:
\begin{itemize}
    \item For all $p$, we have $v_p(d) \in\{ 0,1,3,5\}$.
    \item If $p\mid 3\f$, then $d\in \Qp^{\times 3}$.
\end{itemize}

In the next few lemmas, we compute the local Selmer ratios $c_p(\phi_d)$, $c_p(\phi'_d)$, $c_p(\psi_d)$, and $c_p(\psi'_d)$, for all $p$ and all $d \in \Sigma$.
\begin{lemma}\label{lem:good-places}
  If $d\in\Sigma$ and if $p\nmid 3\f d$, then $c_p(\phi_d) = c_p(\phi'_d) = c_p(\psi_d) = c_p(\psi'_d) = 1$
\end{lemma}

\begin{proof}
    By \cite{schaefer}*{Lem.\ 3.8}, we have $c_p(\phi_d) = \frac{c_p(B_d)}{c_p(A_d)}$, where $c_p(B_d)$ and $c_p(A_d)$ are the local Tamagawa numbers. Since $A_d$ and $B_d$ have good reduction at $p$, this equals $1$. The remaining cases are identical.
\end{proof}

\begin{lemma}\label{lem:d-places}
  If $d\in\Sigma$ and if $p\mid d$, then $c_p(\phi_d) = c_p(\phi'_d) = c_p(\psi_d) = c_p(\psi'_d) =c_p(\eta_d)= 1$
\end{lemma}

\begin{proof}
    Let $\alpha\in\{\phi, \phi',\psi,\psi',\eta\}$. By assumption, $A[\alpha]$ is trivial is a $G_{\Q}$-module. Hence, $A_d[\alpha_d] \simeq\chi_d$, where $\chi_d\:G_\Q\to\F_3\t$ is the quadratic character cutting out $\Qp(\sqrt d)$, which is non-trivial, since $v_p(d)$ is odd. Hence, the image of the Kummer map lies in $H^1(\Qp, \chi_d)$, which by \cite{ShnidmanWeiss}*{Lem.\ 4.6}, is trivial. Similarly, since $d$ is not a square, $A_d[\alpha_d](\Qp) = 0$. Hence, $c_p(\alpha_d) = 1$.
\end{proof}

\begin{lemma}\label{lem:pdiv3f}
If $d \in \Sigma$ and $p\mid 3\f$, we have $c_p(\phi_d) = c_p(\psi'_{d}) = c_p(\theta_d)$ and $c_p(\psi_d) = c_p(\phi'_{d}) = c_p(\widehat\theta_{-3d})$. Here, $\theta_d$ is the $d$-th quadratic twist of $\theta\:E\to E'$.
\end{lemma}
\begin{proof}
The condition that $d\in \mathbb Q_p^{\times 3}$ implies that over $\Q_p$, all the abelian varieties in the above diagram are isogenous to products of elliptic curves. For example, up to $2$-isogenies, we have $A_d \approx E_d\times E'_{-3d}$ and $\phi_d\approx 1\times \theta_d$, and so $c_p(\phi_d) = c_p(\theta_d)$. The other equalities follow by an identical analysis.
\end{proof}

\begin{corollary}\label{cor:all-equal}
    For all $d\in \Sigma$ and for all $p$, we have $c_p(\phi_d) = c_p(\psi'_d)$ and $c_p(\phi'_d) = c_p(\psi_d)$.
\end{corollary}

\begin{proof}
    This follows from the previous three lemmas. 
\end{proof}

\begin{lemma}\label{lem:3-reduction}
    We have $c_3(\phi_d) = 1$ and $c_3(\phi'_d) = 3$ for all $d\in \Sigma$.
\end{lemma}

\begin{proof}
    Since $d$ is a cube in $\Q_3$, we have $c_3(\phi_d) = c_3(\theta_d)$ and $c_3(\phi'_d) = c_3(\widehat\theta_{-3d})$. By \cite{schaefer}*{Lem.\ 3.8}, we have $c(\theta_d) = c(E'_d)/c(E_d)\gamma$, where $\gamma\ii$ is the normalized absolute value of the determinant of the map $\mathrm{Lie}(\mathcal{E}) \to \mathrm{Lie}(\mathcal{E'})$ on tangent spaces of the N\'eron models over $\Z_p$.
    
    Since $3\nmid ab$, $E$ has good reduction over $\Q_3$.  The $3$-torsion point $(0, 0)$ on $E$ reduces to a non-trivial point in $E(\F_3)$, so the reduction is ordinary. It follows from \cite{bkls}*{Thm.\ 10.5} that $c(E'_d)/c(E_d) = 1$ and $\gamma = 1$. We therefore have $c_3(\phi_d) = 1$. For $c_3(\phi_d')$ the argument is similar, except the generator of $\ker(\phi'_d) \simeq \ker(\widehat\theta_{-3d})$ reduces to the identity over $\F_3$ (since over $\overline \F_3$, the kernel is the Cartier dual of $\Z/3\Z$, which is $\mu_3$) and so $\gamma^{-1} = 3$ and $c_3(\phi'_d) = 3$.
\end{proof}

\begin{lemma}\label{lem:infinity}
    We have $c_\infty(\phi_d) = c_\infty(\phi'_d) = c_\infty(\eta_d) = \begin{cases}\frac13& d>0\\ 1&d<0.\end{cases}$
\end{lemma}

\begin{proof}
    Write $B =\frac{A}{\langle P\rangle}$. The numerator of $c_\infty(\phi_d)$ is equal to $\#\im(B_d(\R)\to H^1(\Gal(\C/\R), A_d[\phi_d]))$. Since $\#A_d[\phi_d] = 3$ is odd, we have $H^1(\Gal(\C/\R), A_d[\phi_d]) = 0$, so this numerator is 1. The denominator is $\#A_d[\phi_d](\R)$ which is $3$ if and only if $d$ is a square in $\R$, i.e.\ if $d>0$. The arguments for $\phi_d'$ and $\eta_d$ are identical.
\end{proof}

\begin{lemma}\label{lem:bad-reduction}
    Let $d\in\Sigma$. If $p\nmid 3d$ divides $\f$, then $c_p(\phi_d)$ and $c_p(\phi'_d)$ are as in the following table:
    \begin{table}[h]

\begin{tabular}{|c|c||c|c||c|c|}
	\hline
	\multicolumn{2}{|c||}{}&\multicolumn{2}{|c||}{$p\mid a^3-27b$}&\multicolumn{2}{|c|}{$p\mid b$}\\
	\cline{3-6}
\multicolumn{2}{|c||}{}&$c_p(\phi_d)$&$c_p(\phi'_d)$&$c_p(\phi_d)$&$c_p(\phi'_d)$\\
\hline
$p=2$&\begin{tabular}{c}
     $d\in \Zp^{\times2}$  \\
     $-3d\in \Zp^{\times 2}$ \\
     $d, -3d\notin\Zp^{\times 2}$
\end{tabular}&\begin{tabular}{l}
	$1$\\
	$3$\\
	$1$
\end{tabular} &\begin{tabular}{l}
    $\frac13$\\
	$1$\\
	$1$
\end{tabular}&\begin{tabular}{l}
    $\frac13$\\
    $1$\\
	$1$
\end{tabular} &\begin{tabular}{l}
	$1$\\
	$3$\\
	$1$
\end{tabular}\\
\hline
$p\equiv 1\pmod 3$&\begin{tabular}{l}
     $d\in \Zp^{\times2}$  \\
     $d\notin \Zp^{\times 2}$ 
\end{tabular}&\begin{tabular}{l}
	$3$\\
	$1$
\end{tabular} &\begin{tabular}{l}
	$\frac13$\\
	$1$
\end{tabular}&\begin{tabular}{l}
	$\frac13$\\
	$1$
\end{tabular} &\begin{tabular}{l}
	$3$\\
	$1$
\end{tabular}\\
\hline
$p\equiv 2\pmod 3$&\begin{tabular}{l}
     $d\in \Zp^{\times2}$  \\
     $d\notin \Zp^{\times 2}$ 
\end{tabular}&\begin{tabular}{l}
	$1$\\
	$3$
\end{tabular} &\begin{tabular}{l}
	$\frac13$\\
	$1$
\end{tabular}&\begin{tabular}{l}
	$\frac13$\\
	$1$
\end{tabular} &\begin{tabular}{l}
	$1$\\
	$3$
\end{tabular}\\
\hline
\end{tabular}
\end{table}
\end{lemma}

\begin{proof}
    As before, since $d\in \Sigma$, we have $c_p(\phi_d) = c_p(\theta_d)$ and $c_p(\phi'_d) = c_p(\widehat\theta_{-3d})$. The assumption that $(a, b) = 1$ ensures that $E$ is semistable (as we will see momentarily), and hence has multiplicative reduction (since $p \mid \f$, by assumption). We need to determine whether $E$ has split or non-split multiplicative reduction.
    
    First suppose that $p\ne 2$. We have $E\: y^2 + axy + by = x^3$. If $p\mid b$, then modulo $p$, $E$ has equation
    \[\br{Y -\frac{ax}{2}}\br{Y + \frac{ax}2}= x^3,\]
    where $Y = y + \frac{ax}2$, i.e.\ $E$ has split multiplicative reduction. Similarly, if $p\mid a^3 - 27b$, then modulo $p$, $E$ has equation
    \[Y^2 + \frac{a^2}{12}X^2 = X^3\]
    with $Y = y + \frac12(ax + b)$ and $X = x + \frac{a^2}{9}$. Thus $E$ has split multiplicative reduction if and only if $-3$ is a square in $\mathbb Q_p$, i.e.\ if $p \equiv 1 \pmod{3}$. 
 
    When $p = 2$, we compute that $\frac{-c_4}{c_6} \equiv 1\pmod 8$ when $2\mid b$, and $\frac{-c_4}{c_6} \equiv -3\pmod 8$ when $2\mid a^3-27b$. By Hensel's lemma, $\frac{-c_4}{c_6}$ is a square in $\Q_2$ if and only if it is a square in $\Z/8\Z$. From this, one checks that  if $2\mid b$, then $E$ has split multiplicative reduction, while if $2\mid a^3-27b$, then $E_{-3}$ has split multiplicative reduction.
 
    We see that $E_d$ and $E'_d$ have non-split multiplicative reduction whenever $p\mid b$ and $d\notin \Q_p^{\times 2}$ or $p\mid a^3-27b$ and $-3d\notin \Qp^{\times 2}$. In such cases, $c(\phi_d) = c(\theta_d) = 1$ by \cite[Prop.\ 10.4]{bkls}.
    
    Otherwise, the formula in \emph{loc.\ cit.} gives $c_p(\theta_d) = \frac{v_p(j(E'_d))}{v_p(j(E_d))} =  \frac{v_p(j(E'))}{v_p(j(E))}$. We have 
    \[j(E) = \frac{a^3(a^3-24b)^3}{b^3(a^3-27b)} \hspace{2mm} \mbox{ and } \hspace{2mm} j(E') = \frac{a^3(a^3 + 216b)^3}{b(a^3-27b)^3}.\] 
    Since $p\ne 3$ and $(a, b) = 1$, $p$ can only divide the denominator of these $j$-invariants, and the remaining entries for $c_p(\phi_d)$ are easily computed. To compute the entries for $c_p(\phi'_d)$, we use the fact that
    \[c_p(\phi'_d) = c_p(\widehat{\theta}_{-3d}) = c_p(\theta_{-3d})\ii = c_p(\phi_{-3d})\ii,\]
    so that the values of $c_p(\phi'_d)$ can be computed using the values of $c_p(\phi_{-3d})$.
    \end{proof}

\begin{proof}[Proof of Theorem $\ref{thm:pos-prop}$]
    Let $A$ be the abelian surface defined in \Cref{def:A}, with $m = 1$. Let $\Sigma$ be the set of integers defined above. For $100\%$ of $d\in \Sigma$, $d$ has order $6$ in $\Q(\zeta_3)\t/\Q(\zeta_3)^{\times 6}$. Hence, by \Cref{lem:new-rank-comparison}, we have $\rk E(K_d/\Q)\new = \rk A_d(\Q)$ for such $d$. Thus, it is sufficient to show that $\rk A_d(\Q) = 0$ for a positive proportion of $d$ in $\Sigma$.
    
    By \Cref{cor:principally-polarized} and \Cref{lem:completely-reducible}, we are in the setting of \Cref{prop:ab-surface-pos-prop}. Hence, it is sufficient to show that the set $T$ in the statement of \Cref{prop:ab-surface-pos-prop} has positive density. 
    
    Define a subset $T'\sub \Sigma$ as follows, based on the two possible cases considered in the theorem:
    \begin{enumerate}
        \item Suppose that there exists a prime $q\equiv 2\pmod 3$ with $q\mid a^3-27b$. Then $d\in T'$ if and only if:
        \begin{itemize}
            \item if $p\mid \f$ and $p\ne q$, then 
            \begin{itemize}
                \item if $p = 2$, then either $2\mid d$ or $d, -3d\notin \Z_2^{\times 2}$.
                \item if $p\equiv 1\pmod 3$, then $p\mid d$ or $d\notin\Zp^{\times 2}$.
                \item if $p\equiv 2\pmod 3$, then $p\mid d$.
            \end{itemize}
            \item either $d<0$ and $d\in \Z_q^{\times 2}$, or $d>0$ and $-3d\in \Z_q^{\times 2}$.
        \end{itemize}
        \item Suppose that there exist primes $q_1 \equiv 1\pmod3$ and $q_2\equiv 2\pmod 3$ such that $q_1\mid a^3-27b$ and $q_2\mid b$. Then $d\in T'$ if and only if:
        \begin{itemize}
            \item if $p\mid \f$ and $p\notin \{q_1, q_2\}$, then 
            \begin{itemize}
                \item if $p = 2$, then either $2\mid d$ or $d, -3d\notin \Z_2^{\times 2}$.
                \item if $p\equiv 1\pmod 3$, then $p\mid d$ or $d\notin\Zp^{\times 2}$.
                \item if $p\equiv 2\pmod 3$, then $p\mid d$.
            \end{itemize}
            \item $d\in \Z_{q_1}^{\times 2}$.
            \item either $d<0$ and $d\in \Z_{q_2}^{\times 2}$, or $d>0$ and $-3d\in \Z_{q_2}^{\times 2}$.
        \end{itemize}
        
    \end{enumerate}
    
    By Lemmas \ref{lem:good-places}-\ref{lem:d-places} and \ref{lem:3-reduction}-\ref{lem:bad-reduction}, we have $c(\phi_d) = c(\phi_d') = 1$, for all $d \in T'$.  By \Cref{cor:all-equal}, we see that $c(\psi_d) = c(\psi'_d)=1$ as well. Hence, $T'\sub T$. Since $T'$ has positive density, the result follows from \Cref{prop:ab-surface-pos-prop}.
\end{proof}

\subsection{An explicit example}\label{subsec:example}

Consider the elliptic curve $E\: y^2 + 2xy - y = x^3$ of conductor $35$. Then $E$ satisfies hypothesis $(i)$ of \Cref{thm:pos-prop}. We compute a lower bound on the proportion of $d$ such that $\rk E(K_d/\Q)\new = 0$. In this case, we can assume for simplicity that $d$ is squarefree.\footnote{For elliptic curves with many odd primes $p \equiv 2\pmod 3$ of bad reduction, we must consider sets $T'$ which are not contained in the set of squarefree integers, as in the proof of \Cref{thm:pos-prop}. In this example, we can restrict to just squarefree integers, as there is only one such prime.}

\begin{proposition}\label{prop:explicit example}
    Let $T$ be the set of squarefree integers $d$ such that all of the following hold:
    \begin{itemize}
    \item $d\equiv \pm1,\pm8,\pm10\pmod{27}$;
    \item $d\equiv-1, 0\pmod7$;
    \item if $d< 0$, then $d\equiv \pm 1\pmod 5$; and
    \item if $d>0,$ then $d\equiv \pm 2\pmod 5$
    \end{itemize}
    Then for a proportion of at least $1/18$ elements $d\in T$, we have $\rk E(K_d/\Q)\new = 0$.
\end{proposition}

\begin{proof}
    Observe that $T$ is the set of squarefree integers $d$ such that:
    \begin{itemize}
        \item $d\in \Z_3^{\times 3}$;
        \item either $7\mid d$ or $d\in \Z_7^{\times 3}\setminus \Z_7^{\times 6}$;
        \item if $d<0$, then $d\in \Z_5^{\times 6}$; and
        \item if $d>0$ then $d\in \Z_5^{\times 3}\setminus \Z_5^{\times 6}$
    \end{itemize}

    The set $T$ is contained in the set $T'$ defined in part $(i)$ of the proof of \Cref{thm:pos-prop}. Hence, $T$ is contained in the set in the statement of \Cref{prop:ab-surface-pos-prop}.
    
    Let $\eta, \phi, \phi'$ be the isogenies from the above commutative diagram. We compute the Selmer ratio $c(\eta_d)$ for $d\in T$. By \Cref{lem:d-places}, we have $c_p(\eta_d) = 1$ for all $d\in T$, unless $p\in\{3,5,7,\infty\}$. Moreover,
    \[c_p(\eta_d) = \frac{\#\coker\eta_d(\Q_p)}{\#\ker\eta_d(\Q_p)}\ge \frac1{\#A_d[\eta_d](\Qp)}.\]
    Now $\#A_d[\eta_d](\Qp) = 1$ if and only if $d\notin \Qp^{\times 2}$. If $d\in T$, then $d\notin \Q_7^{\times 2}$. Thus $c_7(\eta_d) \ge 1$. Otherwise, we have $c_3(\eta_d) \ge \frac 13$ and $c_\infty(\eta_d)c_5(\eta_d)\ge \frac13$. Thus, for all $d\in T$, we have $c(\eta_d) \ge \frac19$.
    
    By \cite{ShnidmanWeiss}*{Thm.\ 5.2}, we have $\avg_{d\in T}\#\Sel(\eta_d) \ge 1 + \frac 19$. On the other hand, as in the proof of \Cref{prop:ab-surface-pos-prop} (see also \Cref{rem:explicit-pos-prop}), we have
    \[1 + \frac19\le \avg_{d\in T}\#\Sel(\eta_d)\le \avg_{d\in T}\mathrm{min}_d = 4-\avg_{d\in T}\mathrm{max}_d,\]
    where $\min_d = \min(\#\Sel(\phi_d), \#\Sel(\phi'_d))$ and $\max_d = \max(\#\Sel(\phi_d), \#\Sel(\phi'_d))$. Hence 
    \[\avg_{d\in T}\mathrm{max}_d \le 3-\frac19.\]
    Let $s_0$ be the proportion of $d\in T$ with $\max_d = 1$. Then
    \[s_0 + 3(1-s_0) \le \avg_{d\in T}\mathrm{max}_d \le 3-\frac19\]
    from which it follows that $s_0\ge \frac{1}{18}$.
    
    We see that for a set of $d\in T$ of relative density $\frac 1{18}$, we have $\max_d = 1$, i.e.\ $\#\Sel(\phi_d) = \#\Sel(\phi_d') = 1$. As in the proof of \Cref{prop:ab-surface-pos-prop}, we see that $\rk E(K_d/\Q)\new = 0$ for such $d$.
\end{proof}

\section{Application to Hilbert's 10th problem over pure sextic fields}\label{sec:hilbert}
Before giving the proof of \Cref{thm:hilbert}, we recall from \cite{garcia-fritz-pasten} some facts related to diophantine sets over a ring $R$ (i.e.\ sets characterized as solutions to polynomial equations over $R$) and Hilbert's tenth problem over number fields.

\begin{definition}
    We say that a subset $S\sub R^n$ is \emph{diophantine} over $R$ if there exist integers $k, m$ and polynomials $F_1, \ldots, F_k\in R[x_1, \ldots, x_n, y_1, \ldots, y_m]$ that satisfy the following property: $(a_1, \ldots a_n) \in S$ if and only if there exists an element $(b_1, \ldots, b_m)\in A^m$ such that for every $j = 1, \ldots, k$, we have $F_j(a_1, \ldots, a_n, b_1, \ldots, b_m) = 0$.
\end{definition}

\begin{definition}
An extension $K/F$ is {\it integrally diophantine} if $\O_F$ is diophantine in $\O_K$.  
\end{definition}

It is well-known that if $K/\Q$ is integrally diophantine, then Hilbert's tenth problem has a negative solution over $\O_K$, so we aim to show that many pure sextic fields are integrally diophantine. Such sextic fields contain subfields, and we will use:

\begin{lemma}{\cite{garcia-fritz-pasten}*{Lem.\ 3.1}}\label{lem:transitivity}
The property of being integrally diophantine is transitive in towers of number fields.
\end{lemma}

We will also use the following result of Shlapentokh \cite{Shlapentokh}:

\begin{theorem}\label{thm:shlapentokh}
Let $K/F$ be a finite extension of number fields.  If there exists an elliptic curve $E/F$ such that $\rk E(F) = \rk E(K) > 0$, then $K/F$ is integrally diophantine. 
\end{theorem}

As well as a recent result of Smith \cite{Smith-thesis}:
\begin{theorem}\label{thm:smith}
Let $E/\Q$ be an elliptic curve with $E[2](\Q) \not\simeq \Z/2\Z$. Then for $100\%$ of integers $d$, we have $\rk E_d = (-1)^{\dim_{\F_2}\Sel_2(E_d)}$, where $E_d$ is the $d$-th quadratic twist of $E$. In particular, for $100\%$ of integers $d$, if $E_d$ has even $2$-Selmer rank, then $\rk E(\Q(\sqrt{d})) = \rk E(\Q)$. 
\end{theorem}

\begin{proof}[Proof of \Cref{thm:hilbert}]
Let $E = E_{a,b}$ be an elliptic curve over $\Q$ satisfying the conditions of \Cref{thm:pos-prop} and having positive rank.   We also insist that $E[2](\Q) \not\simeq \Z/2\Z$. For example, we may take $E = E_{4,-5}\colon y^2 +4xy - 5y = x^3$.  \Cref{thm:pos-prop} gives that for a set of positive lower density $\Sigma \subset \Z$ of sixth-power-free integers $d$, the new rank of $E/K_d$ is 0, where $K_d = \Q(\sqrt[6]{d})$. 

\begin{lemma}
For all $d \in \Sigma$, the $2$-Selmer group  $\Sel_2(E_d)$ has even $\F_2$-dimension.
\end{lemma}
\begin{proof}
Let $\theta_d \colon E_d \to E'_d$ be the $3$-isogeny. By the construction in the proof of \Cref{thm:pos-prop}, the Selmer ratio $c(\theta_d)= \prod_p c_p(\theta_d) = \prod_p c_p(\phi_d) = c(\phi_d)$ is equal to $1$ for all $d \in \Sigma$. Indeed, if $p\mid 3\f$, where $\f$ is the conductor of $E$, then $c_p(\theta_d) = c_p(\phi_d) $ by \Cref{lem:pdiv3f}. And if $p\nmid 3\f$, then $c_p(\theta_d) = c_p(\phi_d)=1$ by the same arguments as Lemmas \ref{lem:good-places} and \ref{lem:d-places}. The formula \[\log_3  c(\theta_d) \equiv \dim_{\F_3}\Sel_3(E_d) - \dim_{\F_3}E_d[3](\Q) \pmod{2}\] (see \cite[Prop.\ 49(b)]{elk}) shows that $\dim_{\F_3} \Sel_3(E_d)$ is even for all but finitely many $d \in \Sigma$. By the $3$- and $2$-parity theorems (due to Dokchitser--Dokchitser \cite{DokchisterDokchitserparity} and Monsky \cite{Monsky}, respectively), it follows that $\dim_{\F_2} \Sel_2(E_d)$ is also even. 
\end{proof}
Hence by \Cref{thm:smith}, the rank of $E_d$ is $0$ for $100\%$ of $d \in \Sigma$.  
Let $F = \Q(\sqrt[3]{d})$.  Since $\rk E(\Q(\sqrt{d})) = \rk E(\Q)$ and $\rk E(K_d/\Q)\new = 0$, we see that $0 < \rk E(F) = \rk E(K_d)$, and hence $K_d/F$ is integrally diophantine by \Cref{thm:shlapentokh}. Since every cubic field (or more generally, any number field with at most one complex place) is known to be integrally diophantine, it follows from \Cref{lem:transitivity} that $K_d/\Q$ is integrally diophantine as well. Hence Hilbert's tenth problem has a negative solution over $\O_{K_d}$.      
\end{proof}
\section{Ranks of QM abelian surfaces}\label{sec:QM}

We recall some facts from \cite{ShnidmanWeiss}*{\S 10} and then give a proof of \Cref{thm:qm-pryms}. 

Let $a> b$ be distinct positive integers, and let $f(x) = (x-a^2)(x-b^2)$. Then $y^3 = f(x^2)$ is an affine model of a smooth projective plane quartic curve $C$ that admits a double cover $\pi \colon C \to E$ to the elliptic curve $E \colon y^3 = f(x)$.  Let $A$ be the Prym variety, i.e.\ the kernel of the map $J=\Jac(C)\to E$ induced by Albanese functoriality. The $\zeta_3$-multiplication on $J$ induces $\zeta_3$-multiplication on $A$, so we may speak of the sextic twists $A_d$. In fact, $A_d$ is simply the Prym variety of $C_d \colon y^3 = (x^2 - da^2)(x^2-db^2)$, which covers the elliptic curve $E_d \colon y^3 = (x-da^2)(x-db^2)$. 

Recall that $\pi$ is the descent of $1 - \zeta_3$ to $\Q$, or in other words, a descent of $\sqrt{-3}$.  Note that $A[\pi] \simeq (\Z/3\Z)^2$ is spanned by the rational points $P=(a, 0) - (-a,0)$ and $Q=(b, 0) - (-b, 0)$.

The Prym variety $A$ need not be principally polarized over $\Q$, but it admits a polarization $\lambda \colon A \to \widehat A$ whose kernel is order 4 \cite{MumfordPrym}. By \cite{ShnidmanWeiss}*{Lem.\ 10.4}, the Rosati involution restricts to complex conjugation on the subring $\Z[\zeta_3]\sub \End(A)$. Thus, we are in the setting of \Cref{prop:ab-surface-pos-prop}. 

The endomorphism $[3]\:A\to A$ factors as $[3] = \pi_{-27}\circ\pi$. Hence, for each $d\in \Q\t/\Q^{\times 6}$, 
\[\rk(A_d) \le \dim_{\F_3}\Sel_3(A_d) \le \dim_{\F_3}(\Sel(\pi_d) \+ \Sel(\pi_{-27d})) = \dim_{\F_3}(\Sel(\pi_d)\+\Sel(\widehat{\pi}_d)),\]
where the last equality follows from \Cref{lem:equal-selmer}.

As in \Cref{sec:pos-prop}, for each $d\in \Q\t/\Q^{\times 6}$, the isogeny $\pi_d$ factors as
\[\begin{tikzcd}
                                                                               & \left(\frac{A}{\langle P\rangle}\right)_d \arrow[rd, "\phi'_{d}"]            &          \\
A_d \arrow[ru, "\phi_d"] \arrow[rd, "\psi_d"'] \arrow[r, "\eta_d" description] & \left(\frac{A}{\langle P+Q\rangle}\right)_d \arrow[r, "\eta'_d" description] & A_{-27d} \\
                                                                               & \left(\frac{A}{\langle Q\rangle}\right)_d \arrow[ru, "\psi'_d"']             &         
\end{tikzcd}\]

Let $\f$ be the conductor of $A$, and let $\Sigma$ be the set of squarefree $d\in \Q\t/\Q^{\times 6}$ such that $d,-3d\notin \Q_p^{\times 2}$ for all $p\mid 3\f$. To prove \Cref{thm:qm-pryms}, we first compute the four Tamagawa ratios $c(\phi_d), c(\phi'_d), c(\psi_d), c(\psi'_d)$.

\begin{lemma}\label{lem:tam-ratios}
    Let $\alpha\in \{\phi, \phi', \psi, \psi'\}$. If $d\in\Sigma$ and if $p\nmid 3\infty$, then $c_p(\alpha_d)=1$.
\end{lemma}

\begin{proof}
 If $p\nmid \f$, then since $d$ is squarefree, the result follows from \cite{ShnidmanWeiss}*{Thm.\ 4.7}. If $p\mid\f$, then, by assumption, $d,-3d\notin \Q_p^{\times 2}$. We have $\ker\alpha\simeq \Z/3\Z$, so $\ker\alpha_d(\Q_p) = 0$. By \cite{ShnidmanWeiss}*{Thm.\ 4.6}, $H^1(\Q ,\ker\alpha_d) = 0$ as well. Hence $c_p(\alpha_d) = 1$.
\end{proof}

\begin{lemma}\label{lem:tam-infinity}
    Let $\alpha\in \{\phi, \phi', \psi, \psi'\}$. Then $c_\infty(\alpha_d) = \begin{cases}\frac13 & d>0\\1&d<0. \end{cases}$
\end{lemma}

\begin{proof}
    The proof is identical to \Cref{lem:infinity}.
\end{proof}

Now, from the factorisation $[3] = \pi_d\circ\pi_{-27d}$, we see that
\[c_3(\phi_d)c_3(\phi'_d)c_3(\phi_{-27d})c_3(\phi'_{-27d}) = c_3([3]) = 9.\]
By \cite{ShnidmanWeiss}*{Lem.\ 10.5} and the assumption that $d, -3d\notin\Q_3^{\times 2}$, it follows that each of these four ratios are integers. Hence, exactly two of them are $3$ and two of them are $1$.

\begin{lemma}\label{lem:QM}
    Let $\Sigma'$ be the set of $d\in \Sigma$ such that $c_3(\phi_d) \ne c_3(\phi'_d)$. If $\Sigma'$ has positive density, then for a positive proportion of $d\in\Sigma'$, we have $\rk A_d(\Q) \le 1$.
\end{lemma}

\begin{proof}
    If $d\in \Sigma'$, then by the above discussion, one of $c_3(\phi_d)$ and  $c_3(\phi'_d)$ is $1$ and the other is $3$. By Lemmas \ref{lem:tam-ratios} and \ref{lem:tam-infinity}, it follows that for each $d\in\Sigma'$, one of $c(\phi_d)$ and $c(\phi'_d)$ is $1$ and the other is $3^{\pm1}$. Without loss of generality, we can assume that $c(\phi_d) = 1$. By \cite{ShnidmanWeiss}*{Prop.\ 5.4}, for at least $\frac12$ of $d \in \Sigma'$, we have $\Sel(\phi_d) = 0 = \Sel(\widehat\phi_d)$, and for at least $\frac56$ of $d\in \Sigma'$, we have $\dim_{\F_3} \Sel(\phi'_d) \oplus \Sel(\widehat\phi'_{d}) = 1$.  Thus, for at least $\frac56 - \frac12 = \frac13$ of $d \in \Sigma'$, we have 
\[\dim_{\F_3} \Sel_3(A_d) \leq\dim(\Sel(\pi_d) \+\Sel(\widehat\pi_d))\leq \dim (\Sel(\phi_d) \+ \Sel( \widehat\phi_{d})\+\Sel(\phi'_d) \+ \Sel(\widehat \phi'_{d})) \leq 1,\]
which implies that $\rk A_d(\Q) \leq 1$.
\end{proof}

\begin{proof}[Proof of \Cref{thm:qm-pryms}]
    By \Cref{lem:QM}, it remains to show that if $\Sigma'$ has density $0$, then $\rk A_d(\Q) \le 1$ for a positive proportion of $d\in\Sigma$. So assume that $\Sigma'$ has density $0$, and let $T$ be the set of $d\in\Sigma$ such that $c(\phi_d) = c(\phi'_d) = c(\psi_d) = c(\psi'_d) = 1$. By \Cref{prop:ab-surface-pos-prop}, it is sufficient to show that $T$ has positive density, in which case, for a positive proportion of $d\in \Q\t/\Q^{\times 6}$, we have $\rk A_d(\Q) = 0$.  
    
    First note that if $d\in\Sigma\setminus\Sigma'$ and if $c(\phi_d) = c(\phi'_d) = 1$, then $c(\psi_d) = c(\psi'_d) = 1$. Indeed, $c_3(\pi_d) = c_3(\phi_d)c_3(\phi'_d) = c_3(\psi_d)c_3(\psi'_d)$. Since $d\notin\Sigma'$, we have $c_3(\phi_d) = c_3(\phi'_d)$, and since all four of $c_3(\phi_d),c_3(\phi'_d),c_3(\psi_d),c_3(\psi'_d)$ are either $1$ or $3$, we see that $c_3(\phi_d) = c_3(\phi'_d) = c_3(\psi_d) = c_3(\psi'_d)$. Hence, by Lemmas \ref{lem:tam-ratios} and \ref{lem:tam-infinity}, we have $c(\phi_d) = c(\phi'_d) = c(\psi_d) = c(\psi'_d)$ for all $d\in \Sigma\setminus\Sigma'$.
    
    Assume at first that there exists a positive density of $d\in \Sigma$ such that $c_3(\phi_d) = c_3(\phi'_d) = 3$. For any such $d$, if $d>0$, then $c_\infty(\phi_d) = c_\infty(\phi'_d) = \frac 13$, so $c(\phi_d) = c(\phi'_d) = 1$, and so $T$ has postive density. If $d<0$, then let $k\in \Z$ be such that $k\in \Z_p^{\times 2}$ for all $p\mid 3\f$ and $k<0$. Then $dk\in \Sigma$, $dk>0$, and $c_3(\phi_{dk}) = c_3(\phi'_{dk}) = 3$. Thus $dk\in T$.  A similar analysis shows that if $c_3(\phi_d) = c_3(\phi'_d) = 1$ then $T$ again has positive density. The result now follows from \Cref{prop:ab-surface-pos-prop}.
    \end{proof}
\begin{proof}[Proof of $\Cref{thm:qmcurves}$]
Given \Cref{thm:qm-pryms}, the proof is exactly as in \cite[Thm.\ 1.7]{ShnidmanWeiss}
\end{proof}

\section*{Acknowledgements}
The authors are grateful to Hershy Kisilevsky for helpful comments. They also thank the referees for their very helpful suggestions and comments. The first author was supported by the Israel Science Foundation (grant No.\ 2301/20). The second author was supported by an Emily Erskine Endowment Fund postdoctoral fellowship at the Hebrew University of Jerusalem, by the Israel Science Foundation (grant No.\ 1963/20) and by the US-Israel Binational Science Foundation (grant No.\ 2018250).

\bibliography{bibliography}
\bibliographystyle{alpha} 

\end{document}